\documentclass[12pt,reqno,a4paper]{amsart}
\usepackage[british]{babel}

\usepackage{graphicx}
\usepackage{amscd}
\usepackage{amsmath}
\usepackage{amsfonts}
\usepackage{amssymb}
\usepackage{comment}
\usepackage{color}
\usepackage[usenames,dvipsnames,table,xcdraw]{xcolor}
\usepackage{pdfsync}
\usepackage{bbm}
\usepackage{dsfont}

\usepackage{enumerate}
\usepackage{booktabs}
\usepackage{float}
\usepackage{wrapfig}
\usepackage{caption}
\usepackage{subcaption}
\usepackage{longtable}
\usepackage{listings}
\usepackage[T1]{fontenc}
\usepackage[scaled]{beramono}
\usepackage{tikz}
\usepackage{pgffor}
\usepackage[all]{xy}

\usepackage{anysize}
\usepackage[colorlinks=true]{hyperref}
\definecolor{trp}{rgb}{1,1,1}

\definecolor{red}{rgb}{1,0,.2}

\newtheorem{theorem}{Theorem}[section]
\theoremstyle{plain}

\newtheorem{claim}[theorem]{Claim}

\newtheorem{conjecture}{Conjecture}

\newtheorem{lemma}[theorem]{Lemma}

\newtheorem{problem}{Open Problem}
\newtheorem{prop}[theorem]{Proposition}
\newtheorem{remark}[theorem]{Remark}

\numberwithin{equation}{section}

\newcommand{\N}{\mathbb{N}}

\newcommand{\ii}{\mathbf{i}}
\newcommand{\jj}{\mathbf{j}}

\begin{document}
\title[Distance between natural numbers based on their prime signature]{Distance between natural numbers based on their prime signature}

\author{Istv\'an B Kolossv\'ary}
\address{Istv\'an B Kolossv\'ary, Boston University, Department of Biomedical Engineering, \newline Boston, MA 02215, U. S. A.} \email{ikolossv@bu.edu}

\author{Istv\'an T Kolossv\'ary}
\address{Istv\'an T Kolossv\'ary, The University of St Andrews,  School of Mathematics and Statistics, \newline St Andrews, KY16 9SS, Scotland} \email{itk1@st-andrews.ac.uk}


\begin{abstract}
We define a new metric between natural numbers induced by the $\ell_\infty$ norm of their unique prime signatures. In this space, we look at the natural analog of the number line and study the arithmetic function $L_\infty(N)$, which tabulates the cumulative sum of distances between consecutive natural numbers up to $N$ in this new metric. 

Our main result is to identify the positive and finite limit of the sequence $L_\infty(N)/N$ as the expectation of a certain random variable. The main technical contribution is to show with elementary probability that for $K=1,2$ or $3$ and $\omega_0,\ldots,\omega_K\geq 2$ the following asymptotic density holds
$$
\lim_{n\to\infty}\frac{\big|\big\{M\leq n:\; \|M-j\|_\infty <\omega_j \text{ for } j=0,\ldots,K \big\}\big|}{n} = \prod_{p:\, \mathrm{prime}}\! \bigg( 1- \sum_{j=0}^K\frac{1}{p^{\omega_j}} \bigg)\,.
$$
This is a generalization of the formula for $k$-free numbers, i.e. when $\omega_0=\ldots=\omega_K=k$. The random variable is derived from the joint distribution when $K=1$.

As an application, we obtain a modified version of the prime number theorem. Our computations up to $N=10^{12}$ have also revealed that prime gaps show a considerably richer structure than on the traditional number line.  Moreover, we raise additional open problems, which could be of independent interest. 
\end{abstract}

\maketitle

\thispagestyle{empty}

\section{Introducing distances on the prime grid}

The natural numbers form a totally ordered set, traditionally visualized by plotting them as evenly spaced points on the positive half of the number line. This representation gives a clear idea of their magnitude but not much else. Another possible method of visualizing the natural numbers is called the Ulam spiral \cite{UlamSpiral}, where the numbers are placed in increasing order on the grid points of $\mathbb{Z}^2$ starting with 1 at $(0,0)$ and spiraling outwards, so $2\to(1,0),\, 3\to(1,1),\, 4\to(0,1),\, 5\to(-1,1),\, 6\to(-1,0)$ and so on. Interestingly, prime numbers tend to line up along certain diagonal lines corresponding to specific quadratic polynomials, which could be explained by a conjecture of Hardy and Littlewood \cite{Hardy1923} (if proven to be true).

Another natural, off-the-number-line spatial representation is suggested by the fundamental theorem of arithmetic. Each natural number is a unique grid point in an infinite dimensional coordinate system ``spanned'' by the prime numbers, which we coin the \emph{prime grid}. The position of two natural numbers in this grid leads us to define new distances between them. The goal is to study how this affects the natural analogs of the number line and the distribution of primes. 

\subsection{The prime grid}
The factorization theorem states that every natural number $N$ can be uniquely identified with an infinite sequence $\ii^{N}=(i_1,i_2,\ldots)$ of non-negative integers, called the prime signature of $N$, so that
\begin{equation*}
N= p_1^{i_1}p_2^{i_2}\ldots p_k^{i_k}\ldots \, ,
\end{equation*}
where throughout the paper $\{p_k\}_{k=1}^\infty$ denote the prime numbers in ascending order. The number $1$ is represented by the sequence $\mathbf{0}=(0,0,\ldots)$. This gives a bijection between the positive integers and the space of all infinite sequences that consist only of zeros except for a finite number of positive integers. We think of $\mathbf{0}$ as the origin of a coordinate system whose axis are indexed by the prime numbers and each natural number is a grid point on this \textit{prime grid}. The addition of two signatures $\ii^N+\jj^M$ and also the multiplication by a scalar $n\in\mathbb{N}$ is done coordinate-wise. The former represents multiplication $NM$, while the latter raises $N$ to the $n$th power. 

\begin{remark}
ne can canonically extend the grid to get a module over the integers by allowing negative integers in a signature as well to represent all positive rational numbers. In this case subtraction $\ii^N-\jj^M$, done coordinate-wise, represents division $N/M$. 

Another possible generalization is to consider supernatural numbers or Steinitz numbers, i.e. signatures with possibly infinitely many non-zero coordinates that may take the value $\infty$ as well. However, we do not pursue these directions and consider only signatures of natural numbers henceforth.
\end{remark}

We consider the $\ell_\infty$ norm of the signatures of natural numbers, i.e.
\begin{equation*}
\|N\|_\infty=\|\ii^N\|_\infty=\max \{i_1,i_2,\ldots\}.
\end{equation*}
The norm naturally induces the $\ell_\infty$ metric, also referred to as the Chebyshev distance, which for two natural numbers $N$ and $M$ with signatures $\ii^N$ and $\jj^M$ is defined as
\begin{equation*}
d_{\infty}(N,M) = d_{\infty}(\ii^N,\jj^M) = \max\{|i_1-j_1|,|i_2-j_2|,\ldots\}.
\end{equation*}
Recall, a natural number $N$ is called $k$-free if $\|N\|_\infty<k$. Hence, the metric naturally collects $k$-free numbers: they are enclosed in the ball of radius $k-1$ (centered at $\mathbf 0$). In particular, the unit ball consists of the square-free numbers. We refer to the numbers $M$ for which $\|M\|_\infty=k$ as the Chebyshev contour at distance $k$. It is well-known that the number of $k$-free numbers up to $M$ follow an asymptotic of $M/\zeta(k)+ O(\sqrt[k]{M})$, where $\zeta(k)$ is the Riemann zeta function and $O$ is the usual big-$O$ notation. Hence, the Chebyshev contours follow the asymptotic
\begin{equation}\label{eq:vsz_Ninfty=k}
\lim_{n\to\infty}\frac{|\{M\leq n \text{ such that } \|M\|_\infty =k \}|}{n} = \frac{1}{\zeta(k+1)} - \frac{1}{\zeta(k)}.
\end{equation}
The main technical contribution of the paper is to generalize this asymptotic density to sets concerning consecutive numbers like $\{M\leq n \text{ such that } \|M\|_\infty =k,\, \|M-1\|_\infty =\ell \}$. With this new distance between the natural numbers, we wish to study an analog of the number line and also the distribution of prime numbers.

\subsection{The number trail: an analog of the number line}\label{subsec:NumberTrail}

We focus on a particular object on the prime grid, a zigzag path that starts at $\mathbf 0$ and crisscrosses through every single grid point on the prime grid in the order of the increasing sequence of the natural numbers. We term this path the \textit{number trail} and define an arithmetic function $L_\infty(N)$ tabulating the total length of the number trail up to $N$ using the chosen metric. That is
\begin{equation} \label{def:Linfty} 
L_\infty (N) :=\sum_{M=2}^{N}d_\infty(M,M-1) = \sum_{M=2}^{N} \max\{\|M\|_\infty,\|M-1\|_\infty\}, 
\end{equation}
where the second equality holds, since $M$ and $M-1$ are always coprime. $L(1)=0$ by definition. We think of it as an analog of the traditional number line, where $L(N)=\sum_{M=1}^{N}|M-(M-1)|=N$. The sequence $\{L_{\infty}(N)\}_{N=2}^\infty$ is indexed in the On-line Encyclopedia of Integer Sequences (OEIS) database as sequence \href{https://oeis.org/A334573}{A334573}. 

It follows from \eqref{def:Linfty} that the growth of $L_\infty(N)$ depends on the deterministic sequence
\begin{equation*}
\pmb{\Omega} = \|2\|_\infty,\|3\|_\infty,\|4\|_\infty,\ldots .
\end{equation*}
More precisely, the asymptotic in \eqref{eq:vsz_Ninfty=k} is not enough, but rather the distribution of pairs $(\|M\|_\infty,\|M-1\|_\infty)$ determine $L_\infty$. This is the content of our main technical result.

\begin{theorem}\label{thm:00}
	For any $\omega_0,\omega_1\geq 2$, the following asymptotic density is satisfied:
	\begin{equation}\label{eq:10}
	\lim_{n\to\infty}\frac{\big|\big\{M\leq n:\; \|M\|_\infty <\omega_0 \text{ and } \|M-1\|_\infty <\omega_1\big\}\big|}{n} = \prod_{p:\, \mathrm{prime}}\! \bigg( 1- \frac{1}{p^{\omega_0}} - \frac{1}{p^{\omega_1}}\bigg).
	\end{equation}
	Let $\pi(\omega_0,\omega_1)$ denote the limiting constant in \eqref{eq:10}, moreover, let $N_n$ denote a random integer chosen uniformly at random from the set $\{1,2,\ldots,n\}$.
	
	Then the sequence of random variables $\big\{(\|N_n\|_\infty,\|N_n-1\|_\infty)\big\}_{n}$ has a distributional limit
	\begin{equation*}
	(\|N_n\|_\infty,\|N_n-1\|_\infty) \stackrel{d}{\to} (X_0,X_1),
	\end{equation*}
	with probability mass function $\mathds{P}\big( (X_0,X_1)=(k,\ell) \big)$ equal to
	\begin{equation*}
	\begin{cases}
	\pi(2,2), &\text{if } k=\ell=1, \\
	\pi(2,\ell+1)-\pi(2,\ell), &\text{if } k=1,\ell\geq 2,\\
	\pi(k+1,2)-\pi(k,2), &\text{if } k\geq 2, \ell=1,\\
	\pi(k+1,\ell+1)-\pi(k+1,\ell)-\pi(k,\ell+1)+\pi(k,\ell), &\text{if } k,\ell\geq 2.
	\end{cases}
	\end{equation*}
	In particular, the limit of the expectation of $Y_n:= \max\{\|N_n\|_\infty,\|N_n-1\|_\infty\}$ is equal to
	\begin{equation*}
	C_0:=\lim_{n\to\infty} \mathds{E} Y_n = \mathds{E}\max\{X_0,X_1\}=2.288369512646\ldots.
	\end{equation*}
\end{theorem}

We point out that the formula in~\eqref{eq:10} has previously been shown by Brandes in~\cite{Brandes2015} using a power sieve method. Prior to~\cite{Brandes2015}, related results to~\eqref{eq:10} have only been obtained for $\pmb{\omega}$ with equal coordinates, i.e. tuples of $k$-free numbers. Carlitz~\cite{Carlitz1932} was the first to show that~\eqref{eq:10} holds for $\omega_0=\omega_1=2$. Through a number of intermediate steps, such as \cite{DietmanMarmon14,Heath1984,mirsky1949}, the most general result for $r$-tuples of $k$-free numbers is due to Reuss~\cite{ReussarXiv2014}. A comprehensive list of results with a similar flavor can be found in~\cite{SurveyTsvetkov20019}. The powerful sieve and determinant methods used in these proofs also give bounds on the error term. 

In contrast, the new proof we present for~\eqref{eq:10} to obtain the leading term is a short and completely elementary probabilistic approach. A further advantage of it is that  the density in~\eqref{eq:10} naturally generalises to triplets $\pmb{\omega}=(\omega_0,\omega_1,\omega_2)$. Namely, we prove in Section~\ref{sec:proofs} using the probabilistic approach that for every $\omega_0,\omega_1,\omega_2\geq 2:$
\begin{equation}\label{eq:11}
	\lim_{n\to\infty}\frac{\big|\big\{M\leq n:\; \|M-j\|_\infty <\omega_j \text{ for } j=0,1,2 \big\}\big|}{n} = \prod_{p:\, \mathrm{prime}}\! \bigg( 1- \sum_{j=0}^2\frac{1}{p^{\omega_j}} \bigg)~.
\end{equation} 
Section~\ref{sec:proofs} contains the proof of Theorem~\ref{thm:00} as well. Moreover, we comment on potentially extending to longer $\pmb{\omega}$ sequences in Remark~\ref{rm:00}. From~\eqref{eq:10} and \eqref{eq:11} the limiting joint distributions of $(\|N_n\|_\infty,\|N_n-1\|_\infty)$ and $(\|N_n\|_\infty,\|N_n-1\|_\infty,\|N_n-2\|_\infty)$ can be readily obtained by a standard differencing argument.  

Returning now to our original motivation of determining the asymptotic growth rate of $L_\infty(N)$, from Theorem~\ref{thm:00} it is natural to guess that this should be $cN+o(N)$ for some $0<c<\infty$. Moreover, $L_\infty(N)/N$ is an average of the consecutive values of $\max\{\|M\|_\infty,\|M-1\|_\infty\}$. Thus, with the probabilistic interpretation of $C_0$, it is reasonable to expect that in fact $c=C_0$. This is precisely our main result.

\begin{theorem}\label{thm:01}
	The limit
	\begin{equation}\label{eq:ConjectureLinfty/N}
		\lim_{N\to\infty}\frac{1}{N}L_\infty(N) \;\text{ exists and is equal to }\; C_0. 
	\end{equation}
\end{theorem}

\begin{problem}\label{prob:MagnitudeofError}
	What is the order of magnitude of $|L_{\infty}(N)-C_0N|$?
\end{problem}

From the raw data, we estimated the order of magnitude of the error $|L_{\infty}(N)-C_0N|$ to be $O(N^{0.28})$, considerably smaller than the leading term. Notice that the function $\max\{\|N\|_{\infty},\|N-1\|_{\infty}\}$ of $N$ (and thus $L_{\infty}(N)$) is neither additive nor multiplicative. Hence, well-developed methods from probabilistic number theory are not available to obtain bounds on the error term. Current bounds from the sieve methods also can not account for this drop, indicating that this could be a challenging problem. Figure~\ref{fig:ratio} below shows a list plot of the ratios $L_{\infty}(p_k)/p_k$ for every $10^5$-th prime up to $10^{12}$ plotted against estimates on its fluctuation around $C_0$, see Subsection~\ref{subsec:DirectCompOfLinfty} for details. 
\begin{figure}[H]
	\centering
	\includegraphics[width=1.0\textwidth]{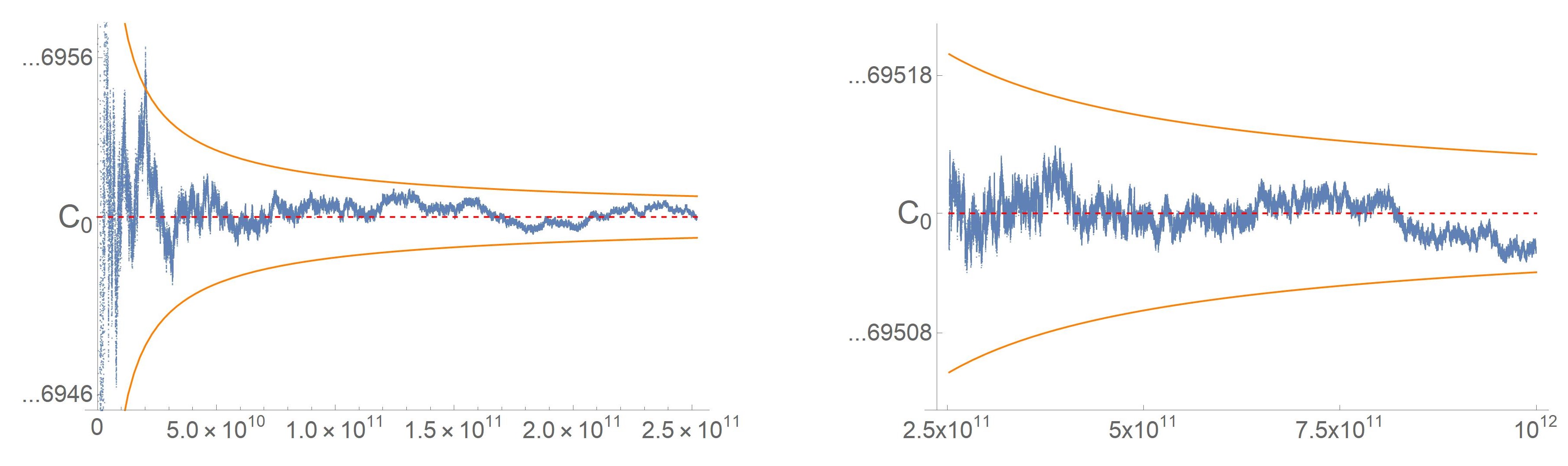}
	\caption{Ratio $L_\infty(p_k)/p_k$ for every $10^5$-th prime up to $10^{12}$. For reference, $C_0$ and $C_0\pm N^{-0.72}$ are plotted as well. Each value on $y$-axis begins with $2.2883\ldots$. }
	\label{fig:ratio}
\end{figure} 

\subsection{Distribution of primes on the number trail}\label{subsec:DistrPrimes}

After establishing the asymptotic growth rate of $L_\infty(N)$, we turn to the effect that the new metric has on the distribution of primes. Let us modify the prime counting function to count the number of primes up to $N$ on the number trail,
$$\pi_\infty(N):= \max \{k:\; L_\infty(p_k)\leq N\}.$$
Of course, $\pi_\infty(L_\infty(p_k))=k$. Combining Theorem~\ref{thm:01} with the prime number theorem, we obtain a modified version for $\pi_\infty(N)$. 
\begin{theorem}[Modified Prime Number Theorem]\label{thm:ModPNT}
	Let $\mathrm{Li}(x)$ denote the offset logarithmic integral function $\int_{2}^{x}1/\ln(y) \mathrm dy$. Then
	\begin{equation*}
	\lim_{N\to\infty}\frac{\pi_\infty(N)}{N/\log N} = \lim_{N\to\infty}\frac{\pi_\infty(N)}{\mathrm{Li}(N)} = \frac{1}{C_0} = 0.43699236\ldots.
	\end{equation*}
\end{theorem}
\begin{proof}
Using Theorem~\ref{thm:01} we can write
\begin{equation*}
\pi_\infty(N)= \max \{k:\; C_0\cdot p_k+o(p_k) \leq N\} = \max \{k:\; p_k \leq N/C_0 +o(p_k)\}.
\end{equation*}
The prime number theorem further implies that
\begin{equation*}
\pi_\infty(N)= \frac{(1+o(1))N/C_0}{\big(1-\frac{\log C_0}{\log N}\big)\log N +o(1)}\,.
\end{equation*}
Dividing by $N/\log N$ and taking limit $N\to\infty$ concludes the proof.
\end{proof}

As in the case on the traditional number line the rate of convergence is very slow and $\mathrm{Li}(N)$ gives a better approximation to $C_0\pi_\infty(N)$ than $N/\log N$. In Table \ref{table:pi_infty(N)} we present the ratios for a few increasing values of $p_k$.
\begin{table}[H]
	\caption{Comparing $\pi_\infty(N)$ to $N/\log N$ and $\mathrm{Li}(N)$, for $N=L_\infty(p_k)$.} \label{table:pi_infty(N)}
	\begin{center}
		\begin{tabular}{c|rcr}
			\toprule
			$k$ & $p_k$ & $\pi_\infty(N)\cdot \log(N) / N$ & $\pi_\infty(N)/\mathrm{Li}(N)$ \\ \midrule
			\rowcolor[HTML]{E8E8E8}
			$10^6$ & $15\,485\,863$ & $0.49053507\ldots$ & $0.46030511\ldots$ \\
			$10^7$ & $179\,424\,673$ & $0.48303924\ldots$ & $0.45721224\ldots$ \\
			\rowcolor[HTML]{E8E8E8}
			$10^8$ & $2\,038\,074\,743$  & $0.47735295\ldots$ & $0.45478434\ldots$ \\
			$10^9$ & $22\,801\,763\,489$ & $0.47294906\ldots$ & $0.45289153\ldots$  \\
			\rowcolor[HTML]{E8E8E8}
			$10^{10}$ &  $252\,097\,800\,623$ & $0.46942719\ldots$ & $0.45136754\ldots$  \\
			$2\cdot10^{10}$ & $518\,649\,879\,439$ & $0.46850132\ldots$ & $0.45096581\ldots$\\
			\rowcolor[HTML]{E8E8E8}
			$3\cdot10^{10}$ &  $790\,645\,490\,053$ & $0.46798418\ldots$ & $0.45074112\ldots$  \\
			\bottomrule
		\end{tabular}
	\end{center}
\end{table}

Our other interest lies in comparing the prime gap functions
\begin{equation*}
\mathcal D_k^1:= L_\infty(p_{k+1})-L_\infty(p_k)  \;\text{ and }\; \mathcal D_k^2:= \mathcal D_{k+1}^1 -\mathcal D_k^1
\end{equation*}
on the number trail to their counterparts $D_k^1=p_{k+1}-p_k$ and $D_k^2=D_{k+1}^1-D_k^1$ on the traditional number line. In Section~\ref{sec:D1andD2}, we highlight some stark differences between them, such as the fact that $\mathcal D_k^1$ and $\mathcal D_k^2$ can take odd values as well.

\subsection{Other metrics}\label{subsec:OtherMetrics}
In our opinion the $\ell_\infty$ metric is the most natural to use, however other metrics could be used as well. Another natural choice could be the $\ell_1$ norm defined by $\|N\|_1=\|\ii^N\|_1=\sum_k i_k,$ which counts the total number of prime factors (with multiplicities) of $N$. This is the well-known additive arithmetic function $\Omega(N)$ in number theory. In particular, the numbers at unit distance from the origin in the $\ell_1$ metric are the prime numbers. The balls with radius $2$ or more don't have such a nice interpretation or asymptotic density as in the $\ell_\infty$ case, though. 

The length of the number trail in this case is
\begin{equation} \label{def:L1}
L_1 (N) = \sum_{M=2}^{N}d_1(M,M-1) = \sum_{M=2}^{N} \|M\|_1 + \|M-1\|_1 = \|N\|_1+2\sum_{M=2}^{N-1}\|M\|_1,
\end{equation}
where the second equality again holds, since $M$ and $M-1$ are always coprime. Thus, the asymptotic growth of $L_1(N)$ is determined by the growth of $\sum_{M\leq N} \Omega(M)$. 
\begin{claim}\label{prop:BoundsOnL1}
For every $N\geq 2$ 
\begin{equation*}\label{eq:BoundL1}
L_1(N)  = 2 N \log\log\! N +2(A+B) N + o(N),
\end{equation*}
where $A$ denotes the Meissel--Mertens constant and $B=\sum_k\big(p_k(p_k-1)\big)^{-1}$. 
\end{claim}
The ingredients for the proof can be found in for example~\cite{TuranMatFizLap34}. We give the short argument here for completeness.
\begin{proof}
	Since $\|N\|_1=O(\log\log\!N)$, it follows from \eqref{def:L1} that $L_1(N)=2\sum_{M\leq N} \Omega(M)+o(N)$. Let $p$ denote a prime and $\alpha$ a positive integer. Then
	\begin{equation}\label{eq:12}
	\sum_{p^\alpha\leq N} \frac{1}{p^\alpha} = \sum_{p\leq N} \frac{1}{p} + \sum_{\substack{p^\alpha\leq N\\ \alpha\geq 2}} \frac{1}{p^\alpha} = \sum_{p\leq N} \frac{1}{p}+ \sum_{p\leq N}\frac{1}{p^2}\frac{1-o(1)}{1-1/p} = \log\log\!N +A+B+o(1).
	\end{equation}
	Furthermore,
	\begin{equation*}
	\sum_{M\leq N} \Omega(M) = \sum_{p^\alpha\leq N} \left[\frac{N}{p^\alpha}\right] = \sum_{p^\alpha\leq N} \frac{N}{p^\alpha} + o(N).
	\end{equation*}
	Combining this with \eqref{eq:12} completes the proof.
\end{proof}
\subsection*{Structure of the paper}
In the remainder of the paper, we only work with the $\ell_\infty$ norm. Section~\ref{sec:proofs} contains all our results on limiting densities in $\pmb{\Omega},$ including the proof of Theorem~\ref{thm:00}.
We prove Theorem~\ref{thm:01} in Section~\ref{sec:Linfty}. Our results about the prime gap functions $\mathcal D_k^1$ and $\mathcal D_k^2$ are presented in Section~\ref{sec:D1andD2}.

An Appendix is included at the end with supplementary material. It includes raw numerical data of the histograms of $\mathcal D^1$ and $\mathcal D^2$ together with a detailed explanation of the Sagemath Python code used to generate the entire data set.

\section{Limiting densities in \texorpdfstring{$\pmb{\Omega}$}{Omega}, proof of Theorem~\ref{thm:00}}\label{sec:proofs}

In this section $p$ is always a prime and let $N=N_n$ denote a random integer chosen uniformly at random from the set $\{1,2,\ldots,n\}$. The signature $\mathbf{I}=\mathbf{I}^N=(I_2,I_3,I_5,\ldots)$ of $N$ is a random infinite dimensional vector, for which
\begin{equation*}
N=\prod_p p^{I_p}.
\end{equation*}
When it is important to indicate that $I_p$ is in the signature of $N$, we write $I_p^N$. It is well-known~\cite[Chapter 1.2]{ArratiaBook} that as $n\to\infty$, the signature $\mathbf{I}$ tends in distribution
\begin{equation*}
\mathbf{I}=(I_2,I_3,I_5,\ldots) \stackrel{d}{\to} \mathbf{Z}=(Z_2,Z_3,Z_5,\ldots),
\end{equation*}
where the $Z_p$ are independent geometric random variables with distribution
\begin{equation*}
\mathds{P}[Z_p=k] = \left(1-\frac{1}{p}\right)\left(\frac{1}{p}\right)^k \quad \text{for } k=0,1,2,\ldots .
\end{equation*}
Moreover, Kubilius~\cite{KubiliusBook} was the first to show that the total variation distance
\begin{equation*}
d_{\mathrm{TV}}\big(\mathcal{L}\big((I_p:\, p\leq b\big), \mathcal{L}\big((Z_p:\, p\leq b)\big)\big)\to 0 \quad\text{if }\; \frac{\log b}{\log n} \to 0.
\end{equation*}
Hence, one can immediately deduce the asymptotic in \eqref{eq:vsz_Ninfty=k}:
\begin{align*}
\lim_{n\to\infty} &\frac{|\{M\leq n \text{ such that } \|M\|_\infty =k \}|}{n} = \lim_{n\to\infty}\mathds{P}[\|N_n\|_{\infty} = k]  \\
& = \lim_{n\to\infty}\mathds{P}[\|N_n\|_{\infty} < k+1] - \mathds{P}[\|N_n\|_{\infty} < k] \\
&= \lim_{n\to\infty} \mathds{P}[\forall\, p<n:\, I_p < k+1] - \mathds{P}[\forall\, p<n:\, I_p < k] \\
&= \mathds{P}[\forall\, p:\, Z_p<k+1] - \mathds{P}[\forall\, p:\, Z_p<k] = \prod_p \mathds{P}[Z_p<k+1] -\prod_p \mathds{P}[Z_p<k]\\
& = \prod_p\bigg( 1- \frac{1}{p^{k+1}}\bigg) - \prod_p\bigg( 1- \frac{1}{p^{k}}\bigg) = \frac{1}{\zeta(k+1)} - \frac{1}{\zeta(k)}\, .
\end{align*}
The main contribution of this section is to determine the asymptotics of the joint distributions
\begin{align*}
& \lim_{n\to\infty} \mathds{P}_n[\|N\|_{\infty} = \omega_0, \|N-1\|_{\infty} = \omega_1] \;\text{ and }\;  \\
& \lim_{n\to\infty}\mathds{P}_n[\|N\|_{\infty} = \omega_0, \|N-1\|_{\infty} = \omega_1, \|N-2\|_{\infty} = \omega_2] \text{ for } \omega_0,\omega_1, \omega_2\geq 1,
\end{align*}
where $\mathds{P}_n$ indicates that $N$ is chosen uniformly at random from $\{1,2,\ldots,n\}$.

We begin with the joint distribution of $(\|N_n\|_\infty,\|N_n-1\|_\infty)$. Let us denote
\begin{align*}
\Pi_n(\omega_0,\omega_1) &:= \mathds{P}_n[\|N\|_{\infty} = \omega_0, \|N-1\|_{\infty} = \omega_1], \\
\Pi_n^{<}(\omega_0,\omega_1) &:= \mathds{P}_n[\|N\|_{\infty} < \omega_0, \|N-1\|_{\infty} < \omega_1],
\end{align*}
and
\begin{equation*}
\pi(\omega_0,\omega_1):= \prod_{p}\! \bigg( 1- \frac{1}{p^{\omega_0}} - \frac{1}{p^{\omega_1}}\bigg).
\end{equation*}

\begin{prop}\label{prop:30}
	The following asymptotics hold for $\omega_0,\omega_1\geq 2$:
	\begin{align*}
	\lim_{n\to\infty} \Pi_n(1, 1) &= \pi(2,2), \quad  \lim_{n\to\infty} \Pi_n(1, \omega_1) = \pi(2,\omega_1+1)-\pi(2,\omega_1), \\
	\lim_{n\to\infty} \Pi_n(\omega_0, 1) &= \pi(\omega_0+1,2)-\pi(\omega_0,2) , \\
	\lim_{n\to\infty} \Pi_n(\omega_0, \omega_1) &= \pi(\omega_0+1,\omega_1+1)-\pi(\omega_0+1,\omega_1)-\pi(\omega_0,\omega_1+1)+\pi(\omega_0,\omega_1). 
	\end{align*}
\end{prop}
\begin{proof}
	Observe that $\Pi_n(1, 1)=\Pi_n^{<}(2,2)$. Moreover, the standard differencing technique implies that for $\omega_0,\omega_1\geq 2$:
	\begin{align*}
	\Pi_n(1, \omega_1) &= \Pi_n^{<}(2,\omega_1+1) - \Pi_n^{<}(2,\omega_1),\;\; \Pi_n(\omega_0, 1) = \Pi_n^{<}(\omega_0+1,2) - \Pi_n^{<}(\omega_0,2), \\
	\Pi_n(\omega_0, \omega_1) &= \Pi_n^{<}(\omega_0+1,\omega_1+1) -\Pi_n^{<}(\omega_0+1,\omega_1) - \Pi_n^{<}(\omega_0,\omega_1+1) + \Pi_n^{<}(\omega_0,\omega_1).
	\end{align*}
	Thus, it is enough to look at the probability $\Pi_n^{<}(\omega_0,\omega_1)$ with $\omega_0,\omega_1\geq 2$:
	\begin{equation*}
	\Pi_n^{<}(\omega_0,\omega_1) = \mathds{P}_n\big[ \forall\, p<n:\, I_p^{N}<\omega_0,\, I_p^{N-1}<\omega_1\big] = \prod_{p<n} \mathds{P}_n\big[ I_p^{N}<\omega_0,\, I_p^{N-1}<\omega_1\big].
	\end{equation*}
	The last equality holds because `divisibility by primes are independent events'. More formally, the Chinese remainder theorem implies that for distinct primes $p$ and $q$ the map
	\begin{equation*}
	\begin{cases}
	\mathbb{Z}/pq\mathbb{Z} \longrightarrow \mathbb{Z}/p\mathbb{Z} \times \mathbb{Z}/q\mathbb{Z} \\
	x \mapsto ( x \!\mod p, x \!\mod q)
	\end{cases}
	\end{equation*}
	is a bijection, see discussion around~\cite[Proposition 1.3.7.]{KowalskiNotes}. Hence, it remains only to show that $\mathds{P}_n\big[ I_p^{N}<\omega_0,\, I_p^{N-1}<\omega_1\big]\to 1-1/p^{\omega_0}-1/p^{\omega_1}$ as $n\to\infty$. Indeed,
	\begin{align*}
	\mathds{P}_n\big[ I_p^{N}&<\omega_0,\, I_p^{N-1}<\omega_1\big] = \mathds{P}_n\big[ I_p^{N}<\omega_0\big] - \mathds{P}_n\big[ I_p^{N}<\omega_0,\, I_p^{N-1}\geq \omega_1\big] \\
	&= \mathds{P}_n\big[ I_p^{N}<\omega_0\big] - \mathds{P}_n\big[ I_p^{N} = 0,\, I_p^{N-1}\geq \omega_1\big] \\
	&= \mathds{P}_n\big[ I_p^{N}<\omega_0\big] - \sum_{\ell=0}^{p-1}\underbrace{\mathds{P}_n\big[ I_p^{N} = 0,\, I_p^{N-1}\geq \omega_1,\, N\equiv \ell\!\! \mod p\big]}_{\neq 0 \;\Longleftrightarrow\; \ell=1} \\
	&= \mathds{P}_n\big[ I_p^{N}<\omega_0\big] - \mathds{P}_n\big[ I_p^{N-1}\geq \omega_1 \big|\, I_p^{N} = 0,\, N\equiv 1\!\! \mod p\big]\cdot \mathds{P}_n\big[ N\equiv 1\!\! \mod p\big] \\
	&= 1-\frac{1}{p^{\omega_0}}+o(1) - \frac{1/p^{\omega_1}+o(1)}{1-(1-1/p)+o(1)} \cdot  \frac{\lfloor n/p \rfloor}{n} \to 1-\frac{1}{p^{\omega_0}}-\frac{1}{p^{\omega_1}}.
	\end{align*}
	In the last equality we used that $\mathds{P}_n\big[ I_p\geq \omega\big]=1/p^{\omega}+o(1)$ and that conditioned on $N\equiv 1 \mod p$, we have $I_p^{N-1}\geq 1$, hence, we divide $1/p^{\omega_1}$ by $1-\mathds{P}_n\big[I_p^{N-1}=0\big]$.
\end{proof}

\begin{proof}[Proof of Theorem~\ref{thm:00}]
	In the proof of Proposition~\ref{prop:30}, we already established the relative density in~\eqref{eq:10} and the probability mass function of the limiting distribution $(X_0,X_1)$. The formula for $C_0= \mathds{E} \max\{ X_0,X_1\}$ is given by:
	\begin{align}
	C_0 =\, &\pi(2,2) + 2\sum_{k=2}^\infty k\,\big(\pi(2,k+1)-\pi(2,k)\big) \label{eq:20} \\
	&+ \sum_{k=2}^\infty \sum_{\ell=2}^\infty \max\{k,\ell\} \big( \pi(k+1,\ell+1)-\pi(k+1,\ell)-\pi(k,\ell+1)+\pi(k,\ell) \big). \nonumber
	\end{align}
	Table~\ref{table:ApproxC0} below contains approximations of the value of $C_0$ by using the first $m$ prime numbers and truncating the sums at $\|\cdot\|_{\infty}\leq n$.
\end{proof}

\begin{table}[H]
	\caption{Approximating $C_0$ by using the first $m$ prime numbers and truncating the sums at $\|\cdot\|_{\infty}\leq n$.} \label{table:ApproxC0}
	\begin{center}
		\begin{tabular}{c|rrrr}
			\toprule
			$m$ & $n=30$ & $n=40$ & $n=50$ & $n=60$  \\ \midrule
			\rowcolor[HTML]{E8E8E8}
			$10^3$ & 2.288361286306563 & 2.288361316070792
			& 2.288361316108944
			& 2.288361316108984
			\\
			$10^4$ & 2.288368990545230
			& 2.288369020309459
			& 2.288369020347611
			& 2.288369020347651
			\\
			\rowcolor[HTML]{E8E8E8}
			$10^5$ & 2.288369450036167
			& 2.288369479800395
			& 2.288369479838548
			& 2.288369479838588
			\\
			$10^6$ & 2.288369480701602
			& 2.288369510465830
			& 2.288369510503983
			& 2.288369510504023
			\\
			\rowcolor[HTML]{E8E8E8}
			$10^7$ & 2.288369482843734
			& 2.288369512607963
			& 2.288369512646115
			& 2.288369512646155
			\\
			$10^8$ & 2.288369482843734
			& 2.288369512607963
			& 2.288369512646115
			& 2.288369512646155
			\\
			\bottomrule
		\end{tabular}
	\end{center}
\end{table}

\subsection{Joint distribution of \texorpdfstring{$(\|N_n\|_\infty,\|N_n-1\|_\infty, \|N_n-2\|_\infty)$}{(|N_n|_infty,|N_n-1|_infty, |N_n-2|_infty)}}

The technique is analogous to the previous case, only the notation and formulas get more involved. Let
\begin{align*}
\Pi_n(\omega_0,\omega_1, \omega_2) &:= \mathds{P}_n[\|N\|_{\infty} = \omega_0, \|N-1\|_{\infty} = \omega_1, \|N-2\|_{\infty} = \omega_2], \\
\Pi_n^{<}(\omega_0,\omega_1, \omega_2) &:= \mathds{P}_n[\|N\|_{\infty} < \omega_0, \|N-1\|_{\infty} < \omega_1, \|N-2\|_{\infty} < \omega_2],
\end{align*}
and also define
\begin{equation*}
\pi(\omega_0,\omega_1, \omega_2):= \prod_{p}\! \bigg( 1- \frac{1}{p^{\omega_0}} - \frac{1}{p^{\omega_1}} - \frac{1}{p^{\omega_2}}\bigg) \;\text{ and }\; \Pi(\omega_0,\omega_1, \omega_2):= \lim_{n\to\infty} \Pi_n(\omega_0, \omega_1,\omega_2).
\end{equation*}

\begin{prop}\label{prop:31}
	The following asymptotics hold for $\omega_0,\omega_1,\omega_2\geq 2$:
	\begin{align*}
	\Pi(1, 1,1) &= \pi(2,2,2), \\
	\Pi(1, 1,\omega_2) &= \pi(2,2,\omega_2+1)-\pi(2,2,\omega_2), \\
	\Pi(1,\omega_1,1) &= \pi(2,\omega_1+1,2)-\pi(2,\omega_1,2) , \\
	\Pi(\omega_0, 1,1) &= \pi(\omega_0+1,2,2)-\pi(\omega_0,2,2) ,
	\end{align*}
	moreover,
	\begin{align*}
	\Pi(1,\omega_1, \omega_2) &= \pi(2,\omega_1+1,\omega_2+1)-\pi(2,\omega_1+1,\omega_2)-\pi(2,\omega_1,\omega_2+1)+\pi(2,\omega_1,\omega_2), \\
	\Pi(\omega_0,1, \omega_2) &= \pi(\omega_0+1,2,\omega_2+1)-\pi(\omega_0+1,2,\omega_2)-\pi(\omega_0,2,\omega_2+1)+\pi(\omega_0,2,\omega_2), \\
	\Pi(\omega_0, \omega_1,1) &= \pi(\omega_0+1,\omega_1+1,2)-\pi(\omega_0+1,\omega_1,2)-\pi(\omega_0,\omega_1+1,2)+\pi(\omega_0,\omega_1,2),
	\end{align*}
	and finally,
	\begin{align*}
	\Pi(\omega_0, \omega_1,\omega_2) &= \pi(\omega_0+1,\omega_1+1,\omega_2+1)-\pi(\omega_0+1,\omega_1+1,\omega_2) \\ &-\pi(\omega_0+1,\omega_1,\omega_2+1)+\pi(\omega_0+1,\omega_1,\omega_2)- \pi(\omega_0,\omega_1+1,\omega_2+1) \\ &+\pi(\omega_0,\omega_1+1,\omega_2)+\pi(\omega_0,\omega_1,\omega_2+1)-\pi(\omega_0,\omega_1,\omega_2).
	\end{align*}
\end{prop}
\begin{proof}
	With the same differencing argument as in the proof of Proposition~\ref{prop:30}, we can express the probabilities $\Pi_n(\omega_0, \omega_1,\omega_2)$ with the probabilities $\Pi_n^{<}(\omega_0,\omega_1, \omega_2)$. Moreover, the same argument implies that
	\begin{equation*}
	\Pi_n^{<}(\omega_0,\omega_1, \omega_2) = \prod_{p<n} \mathds{P}_n\big[ I_p^{N}<\omega_0,\, I_p^{N-1}<\omega_1,\, I_p^{N-2}<\omega_2\big].
	\end{equation*}
	Hence, it is enough to show that for every $\omega_0,\omega_1,\omega_2\geq 2:$ 
	$$\mathds{P}_n\big[ I_p^{N}<\omega_0,\, I_p^{N-1}<\omega_1,\, I_p^{N-2}<\omega_2\big]\to1-\frac{1}{p^{\omega_0}}-\frac{1}{p^{\omega_1}}-\frac{1}{p^{\omega_2}} \;\text{ as } n\to\infty.$$
	The argument is a proper adaptation of the one in the proof of Proposition~\ref{prop:30}: 
	\begin{align*}
	\mathds{P}_n\big[ I_p^{N}&<\omega_0,\, I_p^{N-1}<\omega_1,\, I_p^{N-2}<\omega_2\big] \\
	&= \mathds{P}_n\big[ I_p^{N}<\omega_0,\, I_p^{N-1}<\omega_1\big] - \mathds{P}_n\big[ I_p^{N}<\omega_0,\, I_p^{N-1}<\omega_1,\, I_p^{N-2}\geq \omega_2\big].
	\end{align*}
	We know the first term tends to $1-1/p^{\omega_0}-1/p^{\omega_1}$, thus, it is enough to show that the second term tends to $1/p^{\omega_2}$. The case $p=2$ requires separate treatment.
	
	First assume that $p\geq 3:$
	\begin{align*}
	\mathds{P}_n\big[ I_p^{N}&<\omega_0,\, I_p^{N-1}<\omega_1,\, I_p^{N-2}\geq \omega_2\big] 
	= \mathds{P}_n\big[ I_p^{N} = 0,\, I_p^{N-1} = 0,\, I_p^{N-2}\geq \omega_2\big] \\
	&= \sum_{\ell=0}^{p-1}\underbrace{\mathds{P}_n\big[ I_p^{N} = 0,\, I_p^{N-1} = 0,\, I_p^{N-2}\geq \omega_2,\, N\equiv \ell\!\! \mod p\big]}_{\neq 0 \;\Longleftrightarrow\; \ell=2} \\
	&= \mathds{P}_n\big[ I_p^{N-2}\geq \omega_2 \big|\, I_p^{N} = 0,\, I_p^{N-1} = 0,\, N\equiv 2\!\! \mod p\big]\cdot \mathds{P}_n\big[ N\equiv 2\!\! \mod p\big] \\
	&= \frac{1/p^{\omega_2}+o(1)}{1-(1-1/p)+o(1)} \cdot  \frac{\lfloor n/p \rfloor}{n} \to \frac{1}{p^{\omega_2}}.
	\end{align*}
	
	Now assume that $p=2:$
	\begin{equation*}
	\mathds{P}_n\big[ I_p^{N}<\omega_0,\, I_p^{N-1}<\omega_1,\, I_p^{N-2}\geq \omega_2\big] 
	= \mathds{P}_n\big[ 1\leq I_p^{N} < \omega_0,\, I_p^{N-1} = 0,\, I_p^{N-2}\geq \omega_2\big].
	\end{equation*}
	First consider the case $\omega_0=2:$
	\begin{align*}
	\mathds{P}_n&\big[I_p^{N}=1,\, I_p^{N-1} = 0,\, I_p^{N-2}\geq \omega_2\big] = \mathds{P}_n\big[I_p^{N}=1,\, I_p^{N-1} = 0,\, I_p^{N-2}\geq \omega_2,\, N\equiv 0\!\! \mod p\big] \\
	&= \mathds{P}_n\big[ I_p^{N-2}\geq \omega_2 \big|\, I_p^{N} = 1,\, I_p^{N-1} = 0,\, N\equiv 0\!\! \mod p\big] \\
	&\qquad\cdot \mathds{P}_n\big[ I_p^{N} = 1,\, I_p^{N-1} = 0 \big|\, N\equiv 0\!\! \mod p\big]\cdot \mathds{P}_n\big[ N\equiv 0\!\! \mod p\big] \\
	&= \frac{1/p^{\omega_2}+o(1)}{1-(1-1/p)(1+1/p)+o(1)} \cdot \frac{(1-1/p)1/p+o(1)}{1-(1-1/p)+o(1)} \cdot \frac{\lfloor n/p \rfloor}{n} \to \frac{1}{p^{\omega_2}},
	\end{align*}
	where we also used that $1-1/2=1/2$. Finally, assume $\omega_0>2:$ 
	\begin{align*}
	\mathds{P}_n&\big[I_p^{N}=1,\, I_p^{N-1} = 0,\, I_p^{N-2}\geq \omega_2\big]  \\
	&= \mathds{P}_n\big[I_p^{N}=1,\, I_p^{N-1} = 0,\, I_p^{N-2}\geq \omega_2\big] +\mathds{P}_n\big[2\leq I_p^{N}<\omega_0,\, I_p^{N-1} = 0,\, I_p^{N-2}\geq \omega_2\big].
	\end{align*}
	The second probability is equal to $0$. Indeed, both $N$ and $N-2$ can not be divisible by $4$ (recall $\omega_2\geq 2$). The other term is exactly the same as in the previous point, where $\omega_0=2$. This concludes the proof.
\end{proof}

\begin{remark}\label{rm:00}
Without checking all the details, we believe the asymptotic density
\begin{equation*}
	\lim_{n\to\infty}\frac{\big|\big\{M\leq n:\; \|M-j\|_\infty <\omega_j \text{ for } j=0,1,2,3 \big\}\big|}{n} = \prod_{p:\, \mathrm{prime}}\! \bigg( 1- \sum_{j=0}^3\frac{1}{p^{\omega_j}} \bigg) 
\end{equation*}
still holds for all quadruplets $\omega_0,\omega_1,\omega_2,\omega_3\geq 2$. The proof goes through without difficulty, except that $p=3$ needs to be handled separately as well (besides $p=2$). Even for $\omega_0=\omega_1=\omega_2=\omega_3= 2$ we get the correct density of 0, agreeing with the fact that four consecutive numbers can never all be square-free. See Section~\ref{subsec:21} for details on these `forbidden words' in $\pmb{\Omega}$. 

For sequences with length $k\geq 5$, the asymptotic density clearly can not be equal to $\prod_{p} \big( 1- \sum_{j=0}^{k-1}p^{-\omega_j} \big)$ in general. For sequences of positive density, additional multiplicative factors depending on the sequence can be expected to appear coming from considerations about small primes. 
\end{remark}

\subsection{Forbidden words in \texorpdfstring{$\pmb{\Omega}$}{Omega}}\label{subsec:21}
We consider ${\pmb{\Omega}}$ as an infinite sequence of letters from the alphabet $\mathcal{A}=\mathbb N$. Let $\pmb{\omega}=\omega_1,\ldots,\omega_n$ denote a word of length $|\pmb{\omega}|=n$ from $\mathcal{A}$. There is a set $\mathcal F$ of forbidden words which never appear in $\pmb{\Omega}$. Observe that any subsequence of consecutive symbols of length $2^n$ must contain at least one element with $\|\cdot\|_\infty\geq n$ for any $n>1$. This defines 
\begin{equation}\label{eq:def_ForbiddenWords}
\mathcal F = \bigcup_{n=1}^\infty \mathcal{F}_n = \bigcup_{n=1}^\infty \big\{\pmb{\omega}:\; |\pmb{\omega}|=2^{n+1},\, \omega_i<n+1 \text{ for } 1\leq i\leq 2^{n+1}\big\}.
\end{equation} 
It describes additional structure present in the sequence $\pmb{\Omega}$. Though, it does not influence the value of $C_0$ because it depends only on two consecutive symbols and the shortest forbidden word has length four. 
\begin{problem}\label{prob:ForbiddenWords}
Does every finite length $\pmb{\omega}\notin\mathcal{F}$ appear in $\pmb{\Omega}$, which does not contain any $\pmb{\tau}\in\mathcal{F}$ as a subsequence of consecutive symbols?
\end{problem}
We believe the answer is affirmative. Here is a simple argument to try to find a specific $\pmb{\omega}$ in $\pmb{\Omega}$. Consider an arbitrary $\pmb{\omega}=\omega_1,\ldots,\omega_n\notin \mathcal F$, which does not contain any $\pmb{\tau}\in\mathcal{F}$ as a subsequence of consecutive symbols. Choose any $n$ distinct primes $\mathbf p= (p_1,\ldots,p_n)$ and look at the system of congruences
\begin{equation}\label{eq:CongSystem}
x+i-1 \equiv 0\mod p_i^{w_i}, \;\;i=1,\ldots,n.
\end{equation}
The Chinese remainder theorem implies that there exists a unique $x$ between $1$ and $M=\prod_{i=1}^n p_i^{w_i}$ which satisfies \eqref{eq:CongSystem}. Then of course $\|x+i-1\|_\infty\geq \omega_i$. If all are equalities, then we found $\pmb{\omega}$ in $\pmb{\Omega}$. If not, then we can try with $x+kM$ for some positive integer $k$, since all such numbers satisfy \eqref{eq:CongSystem}. It is unclear whether such a $k$ always exists. It need not exist for all choices of $\mathbf p$. For example, $\pmb{\omega}=1,1,1$ never appears in $\pmb{\Omega}$ with $\mathbf p=(2,3,5)$, but when $\mathbf p=(5,2,7)$ it does with $x=5$. For illustration, we give some non-trivial examples of $\pmb{\omega}$ in $\pmb{\Omega}$ in Table \ref{table:FindWordInOmega}. 

\begin{table}[H]
	\caption{Selected non-trivial $\pmb{\omega}$ with their place of appearance in $\pmb{\Omega}$} \label{table:FindWordInOmega}
	\begin{center}
		\begin{tabular}{l|rrr}
			\toprule
			$\pmb{\omega}$ & $x+kM$ & $k$ & $\mathbf p$ \\ \midrule
			\rowcolor[HTML]{E8E8E8}
			$17,30$ & $27\,699\,975\,238\,617\,792\,512$ & $1$ & $(2,3)$ \\
			$1,15,3,14$ & $18\,890\,469\,353\,465\,057\,219\,498$ & $7$ & $(2,3,5,7)$  \\
			\rowcolor[HTML]{E8E8E8}
			$1, 2, 2, 1, 3, 5, 2, 1$ & $93\,377\,215\,627\,231\,323$ & $16$ & $(3, 2, 5, 7, 11, 13, 17, 19)$ \\
			\bottomrule
		\end{tabular}
	\end{center}
\end{table}

We conjecture a stronger statement claiming that all such $\pmb{\omega}$ have a strictly positive limiting density in $\pmb{\Omega}$.

\begin{conjecture}\label{conj:LimitDens}
	Assume $\pmb{\omega}=(\omega_0,\omega_1,\ldots,\omega_k)\notin\mathcal{F}$ is a finite length word which does not contain any $\pmb{\tau}\in\mathcal{F}$ as a subsequence of consecutive symbols. Then the limit
	\begin{equation*}
	\lim_{n\to\infty}\frac{\big|\big\{M\leq n:\; \|M-j\|_\infty =\omega_j \text{ for } j=0,1\ldots,k \big\}\big|}{n} 
	\end{equation*}
	exists and is strictly positive.
\end{conjecture}

\section{Asymptotic growth of \texorpdfstring{$L_\infty(N)$}{L_infty}, proof of Theorem~\ref{thm:01}  }\label{sec:Linfty}

Recall, $L_\infty (N) = \sum_{M=2}^{N} \max\{\|M\|_\infty,\|M-1\|_\infty\}$ and the notation
\begin{equation*}
\Pi_N(\omega_0,\omega_1) = \frac{\big|\big\{ M\leq N:\; \|M\|_{\infty} = \omega_0, \|M-1\|_{\infty} = \omega_1 \big\}\big|}{N}.
\end{equation*}
\begin{proof}[Proof of Theorem~\ref{thm:01}]
 It is enough to approximate $L_\infty (N)$ from below with the sequence
\begin{equation*}
L_k(N):= \sum_{M=2}^{N} \max\{\|M\|_k,\|M-1\|_k\},
\end{equation*}
where $\|M\|_k:=\min\{\|M\|_\infty,k\}$. On one hand, the limit 
\begin{equation}\label{eq:30}
C_k:= \lim_{N\to\infty}\frac{1}{N} L_k(N) = \sum_{\omega_0=1}^k \sum_{\omega_1=1}^k \max\{\omega_0,\omega_1\}  \lim_{N\to\infty} \Pi_N(\omega_0,\omega_1)
\end{equation}
exists, because the value of $\lim_{N\to\infty} \Pi_N(\omega_0,\omega_1)$ is well-defined from Proposition~\ref{prop:30}. Moreover, the sequence $C_k$ is non-decreasing and is bounded from above by
\begin{equation*}
\lim_{N\to\infty} \frac{1}{N} L_\infty(N) \leq  \lim_{N\to\infty} \frac{2}{N} \sum_{M=2}^N \|M\|_{\infty} \stackrel{\eqref{eq:vsz_Ninfty=k}}{=} 2 \sum_{\ell=1}^\infty \ell \Big(\frac{1}{\zeta(\ell+1)}-\frac{1}{\zeta(\ell)}\Big) = 
2\sum_{\ell=1}^{\infty} \Big(1-\frac{1}{\zeta(\ell)}\Big),
\end{equation*}
which is an absolutely convergent series. Hence, $C_k$ has a limit and comparing \eqref{eq:30} with \eqref{eq:20} we see that $\lim_{k\to\infty}C_k=C_0$.

On the other hand, there is an absolute constant $c>0$ for which 
\begin{equation*}
\frac{L_k(N)}{N}\leq \frac{L_\infty(N)}{N} \leq \frac{L_k(N)}{N} +2 \sum_{\ell>k}\sum_{p:\, \mathrm{prime}} \frac{\ell}{p^\ell} \leq \frac{L_k(N)}{N} + \sum_{p:\, \mathrm{prime}} \frac{c}{p^{k+1}}
\end{equation*}
First letting $N\to\infty$ and then $k\to\infty$, we conclude that $\lim_{N\to\infty}L_\infty(N)/N= C_0$.
\end{proof}

\subsection{Direct computation of \texorpdfstring{$L_\infty$}{L_infty}}\label{subsec:DirectCompOfLinfty}
In order to analyze the distribution of prime gaps $\mathcal{D}_k^1= L_\infty(p_{k+1})-L_\infty(p_k)$, we kept track of the values of $L_\infty$ for all primes. See Appendix~\ref{sec:Code} for the complete program code with explanations. To obtain the list plot in Figure~\ref{fig:ratio}, we calculated the ratios $L_\infty(p_k)/p_k$ for every $10^5$-th prime between $1$ and $10^{12}$. This gave a set of $376\,079$ data points. The minimal value is $2.28836250$ and the maximal is $2.288371417,$ giving a difference of $8.9\times 10^{-6}$. The list plot revealed that the fluctuations of the ratios diminished quite rapidly around $C_0$ from Theorem~\ref{thm:00}.

To get an estimate on the order of magnitude of the error term, we calculated
\begin{equation*}
\alpha(N):=\frac{\log |L_{\infty}(N)-C_0\cdot N|}{\log N}
\end{equation*}
for our data points, and obtained the list plot shown in Figure~\ref{fig:errorestimate}. The maximum value is $0.298731$, obtained at $848\,321\,917$. The last place where it exceeds $0.28$ is at $20\,571\,786\,113$. The plot shows that most frequently the value lies around $0.22$, but this may be misleading, since $N=10^{12}$ is not too large. We do conjecture that $L_{\infty}(N)=C_0\cdot N +O(N^{0.28})$.
\begin{figure}[H]
	\centering
	\includegraphics[width=0.92\textwidth]{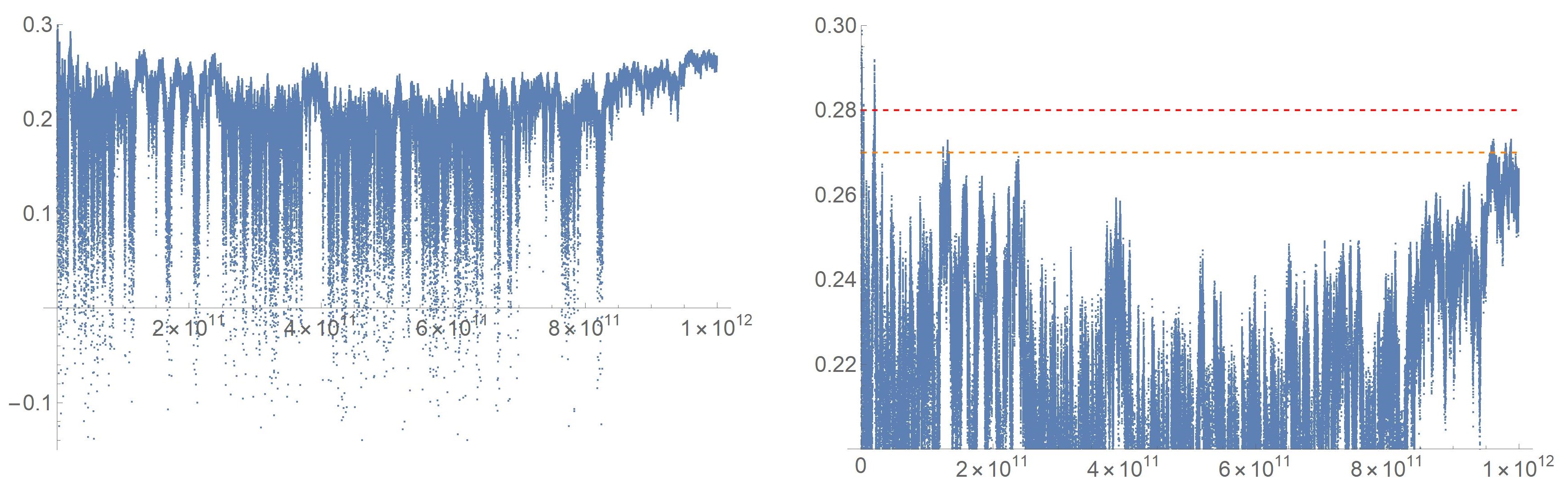}
	\caption{Value of $\alpha(p_k)$ for every $10^5$-th prime between 1 and $10^{12}:$ full-range of values on left, restricting to $\alpha(p_k)\in[0.2,0.3]$ on right.}
	\label{fig:errorestimate}
\end{figure}

\section{Prime gaps on the number trail}\label{sec:D1andD2}

On the traditional number line the most widely used representation for the distribution of prime numbers are the prime gap functions
\begin{equation*}
D_k^1=p_{k+1}-p_k \;\text{ and }\; D_k^2=D_{k+1}^1-D_k^1.
\end{equation*}
Especially $D_k^1$ has continuously received immense attention with results establishing large gaps between primes \cite{Taoetal2014arXiv, West31}, small gaps~\cite{PintzI,PintzII2010}, and limit points of $D^1_k/\log p_k$ \cite{Banks2016,Pintz2018Gaps}. Zhang~\cite{Zhang2014BoundedGaps} made a big breakthrough by proving that $D^1_k$ was bounded from above by a constant for infinitely many $k$. The original constant of $7\times 10^7$ has been greatly reduced by work of the Polymath Project~\cite{Polymath1,Polymath2} and Maynard \cite{Maynard2015}. This list only gives a glimpse, it is far from being exhaustive.

On the other hand, there are long-standing conjectures which are still open today. The twin prime conjecture asserts that $D^1_k=2$ for infinitely many $k$. Polignac's conjecture is even stronger, stating that for any positive even number $N$ there are infinitely many $k$ such that $D^1_k=N$.

Here we propose an alternative approach to study the distribution of prime numbers by looking at the prime gaps along the number trail. Analogous to the prime gap functions $D^1$ and $D^2$, we define the differences $\mathcal D^1$ and $\mathcal D^2$ along the number trail with $L_\infty$ to be
\begin{equation*}\label{def:PrimegapOnNumbertrail}
\mathcal D_k^1:= L_\infty(p_{k+1})-L_\infty(p_k)  \;\text{ and }\; \mathcal D_k^2:= \mathcal D_{k+1}^1 -\mathcal D_k^1.
\end{equation*}
Histograms already reveal stark differences between $D^1,D^2$ and $\mathcal D^1, \mathcal D^2$. Figure~\ref{fig:Histogram_D1_DD1} shows the histograms of $D^1$ and $\mathcal D^1$ side-by-side taking prime numbers up to $N\leq 10^{12}$, while Figure~\ref{fig:Histogram_D2_DD2} shows $D^2$ and $\mathcal D^2$. Appendix~\ref{sec:data} contains tables of the numerical data used in the figures.

\begin{figure}[h]
	\centering
	\begin{subfigure}[b]{0.49\textwidth}
		\includegraphics[width=0.957\textwidth]{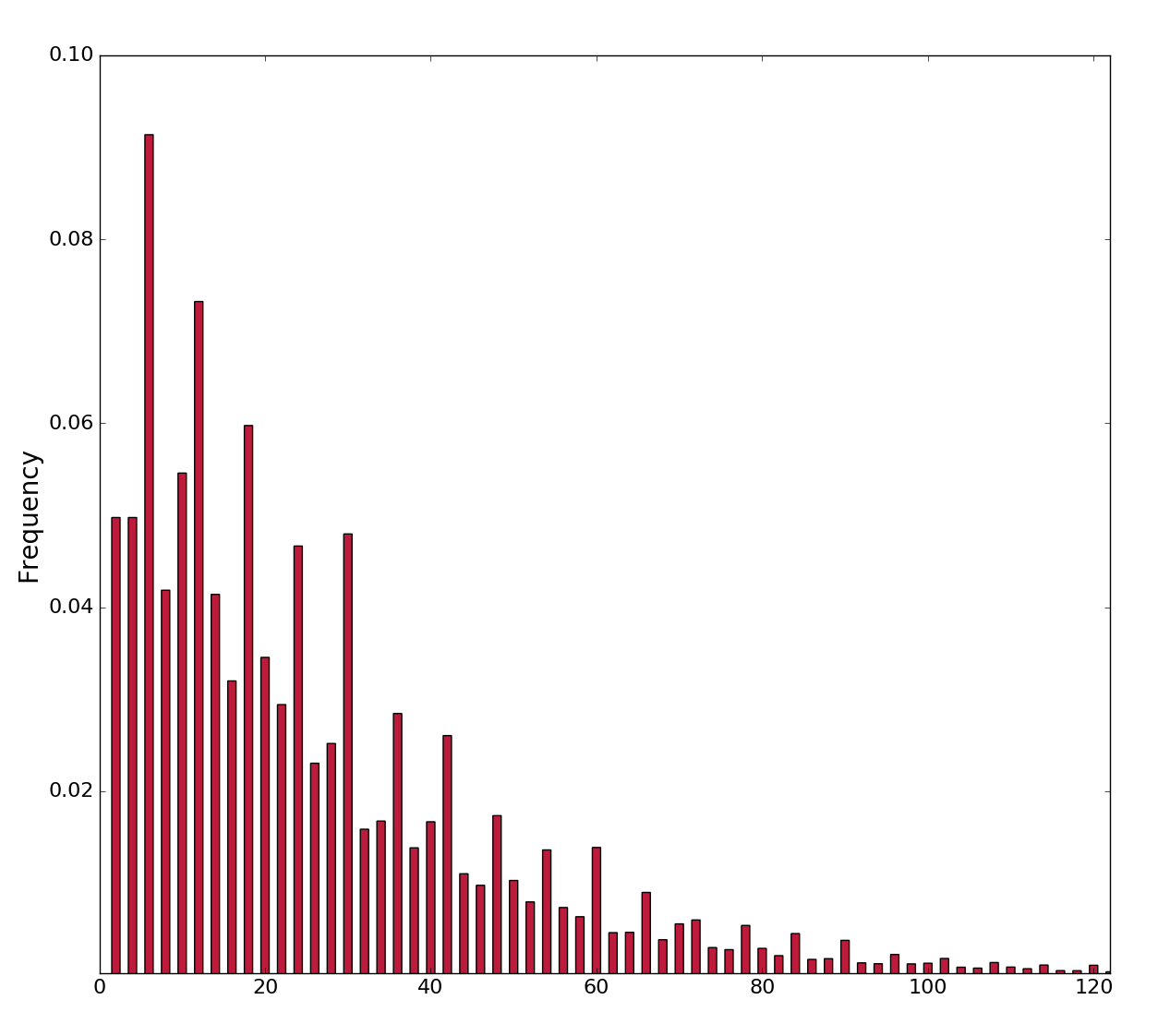}
		\caption{$D_k^1=p_{k+1}-p_k$}
		\label{fig:Hist_D1_1e12}
	\end{subfigure}
	\begin{subfigure}[b]{0.49\textwidth}
		\includegraphics[width=0.98\textwidth]{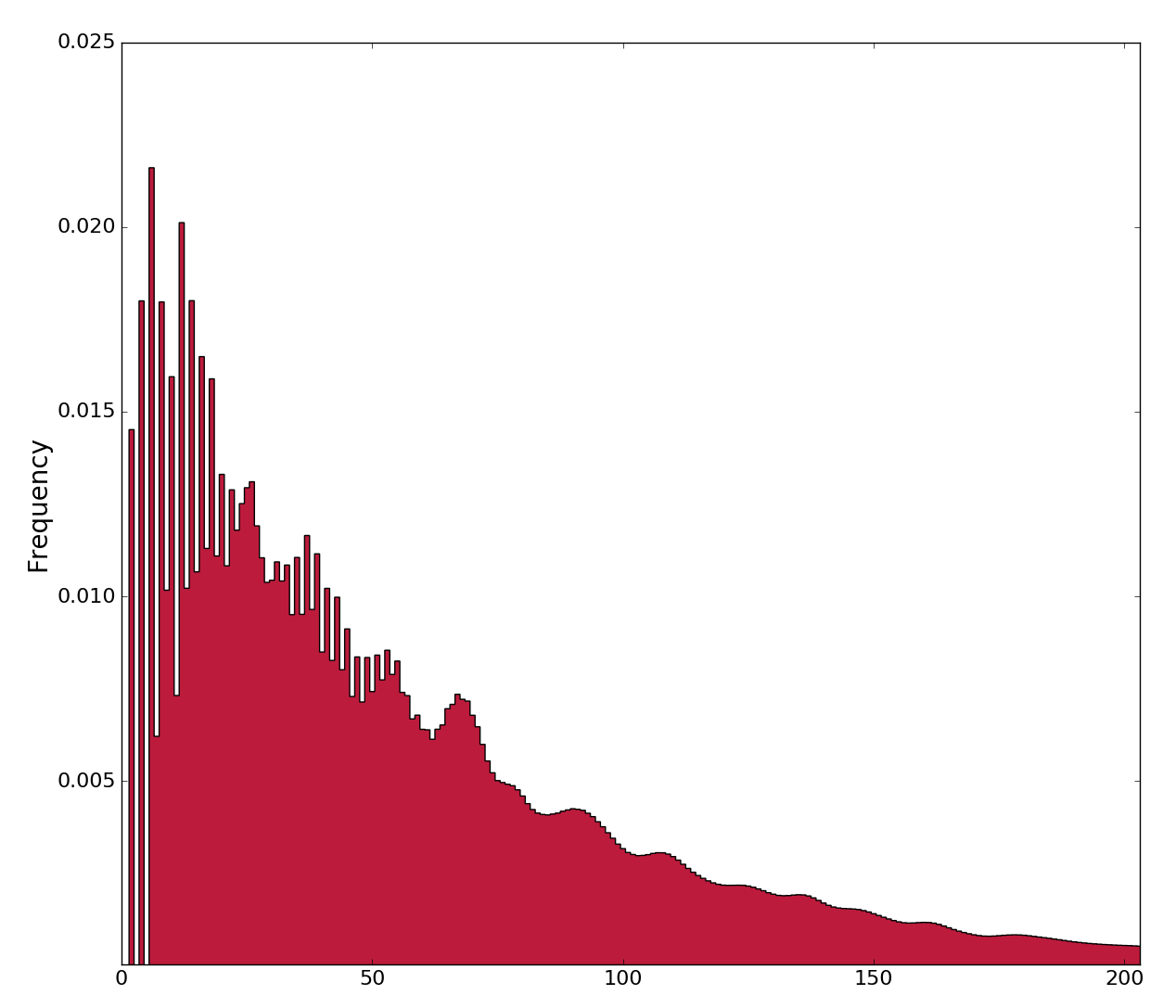}
		\caption{$\mathcal D_k^1= L_\infty(p_{k+1})-L_\infty(p_k)$}
		\label{fig:Hist_D2_1e12}
	\end{subfigure}
	\caption{Histogram of first order differences between consecutive prime numbers up to $N\leq 10^{12}$ on the number line (left) and on the number trail (right). }\label{fig:Histogram_D1_DD1}
\end{figure}

The semi-regular spiked structure of the histograms of $D^1$ and $D^2$ has been shown~\cite{Ares2006} to be attributed to the fact that for every prime number $p$, except 2, $p=\pm 1\!\mod 6$. These spikes are replaced by an intricate, non-repeating fine structure on the  $\mathcal D^1$ and $\mathcal D^2$ histograms with numerous, differently shaped local peaks. Most notably, as becomes better visible in Figure~\ref{fig:Histogram_DD2}~(B) the cap of the $\mathcal D^2$ histogram is not the highest peak, it has two tiny local maxima at $\pm1$, but there are two more, symmetrical and significantly higher maxima at $\pm5$ and $\pm7$the cap of the $\mathcal D^2$ histogram is not the highest peak, it has two tiny local maxima at $\pm1$, but there are two more, symmetrical and significantly higher maxima at $\pm5$ and $\pm7$.

Equally apparent difference between $D^1, D^2$ and $\mathcal D^1, \mathcal D^2$ is that the former only take even values (except for the gap between 2 and 3), whereas the latter can also take odd values.

\begin{figure}[h]
	\centering
	\begin{subfigure}[b]{0.49\textwidth}
		\includegraphics[width=0.95\textwidth]{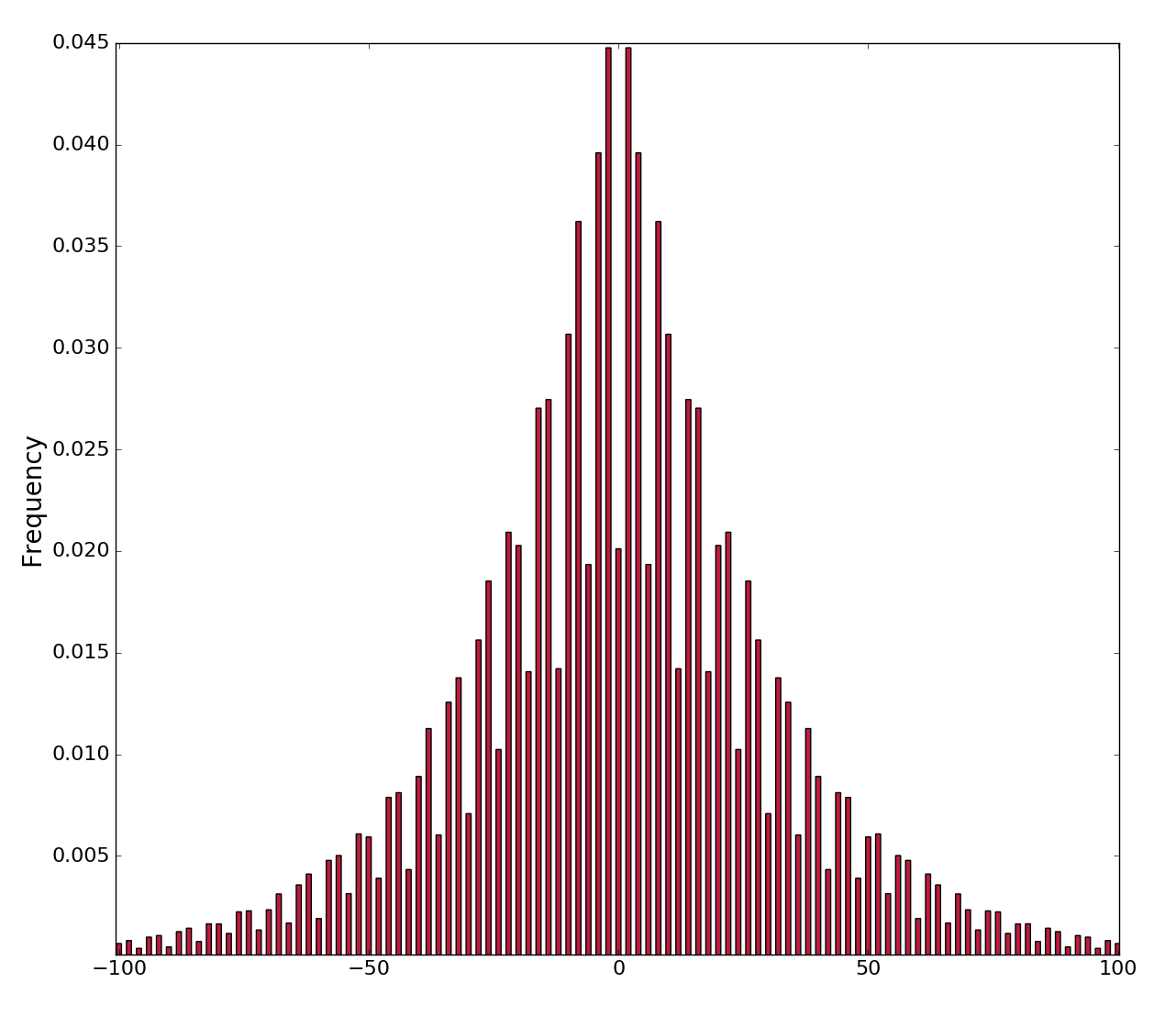}
		\caption{$ D_k^2= D_{k+1}^1- D_k^1$}
		\label{fig:Hist_DD1_1e12}
	\end{subfigure}
	\begin{subfigure}[b]{0.49\textwidth}
		\includegraphics[width=0.95\textwidth]{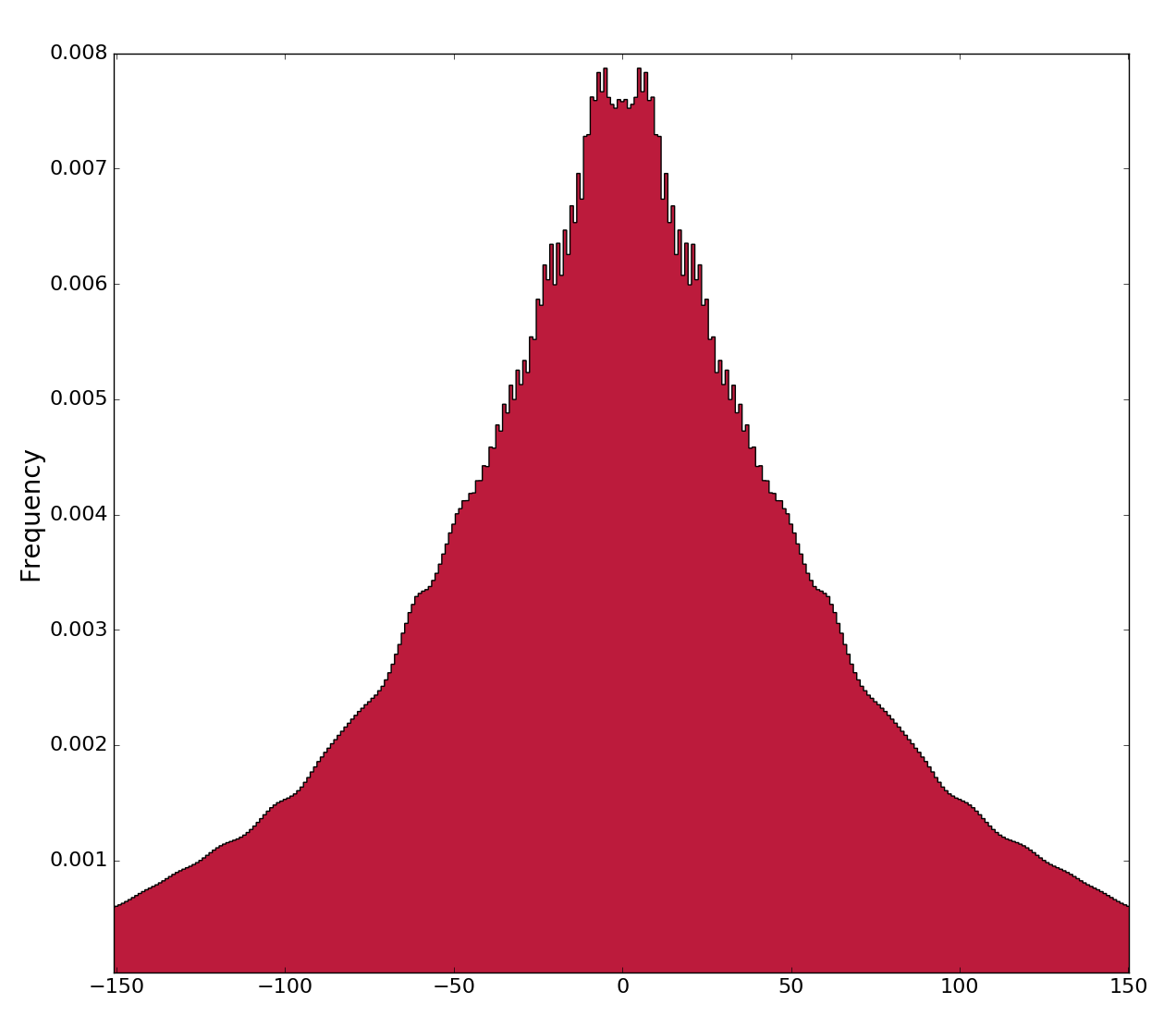}
		\caption{$\mathcal D_k^2=\mathcal D_{k+1}^1-\mathcal D_k^1$}
		\label{fig:Hist_DD1_1e12_zoom}
	\end{subfigure}
	\caption{Histogram of second order differences between consecutive prime numbers up to $N\leq 10^{12}$ on the number line (left) and on the number trail (right).}
	\label{fig:Histogram_D2_DD2}
\end{figure}

\begin{figure}[h]
	\centering
	\begin{subfigure}[b]{0.51\textwidth}
		\includegraphics[width=0.95\textwidth]{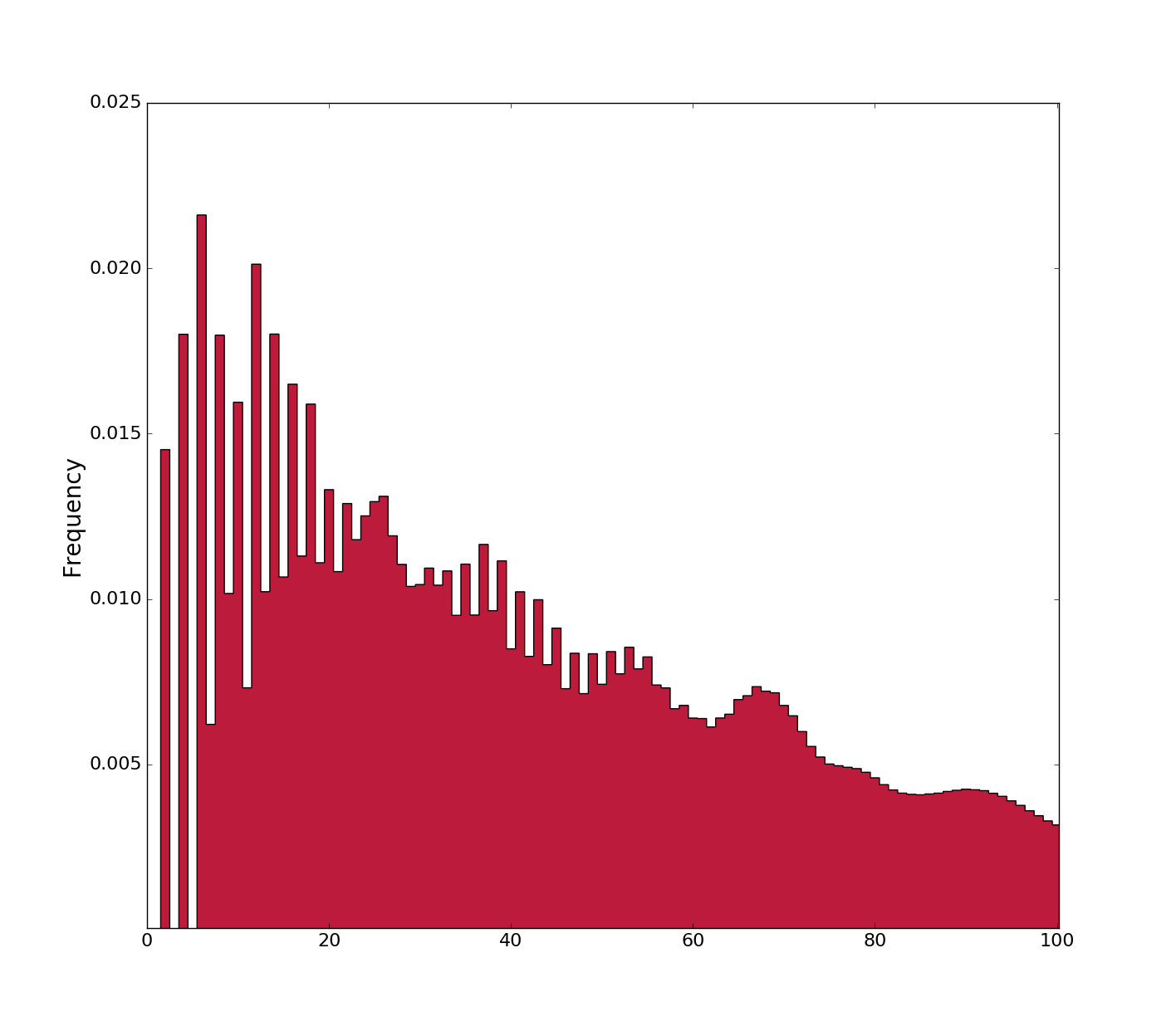}
		\caption{$\mathcal D_k^1$, close-up}
		\label{fig:Hist_DD2_1e12}
	\end{subfigure}
	\begin{subfigure}[b]{0.48\textwidth}
		\includegraphics[width=\textwidth]{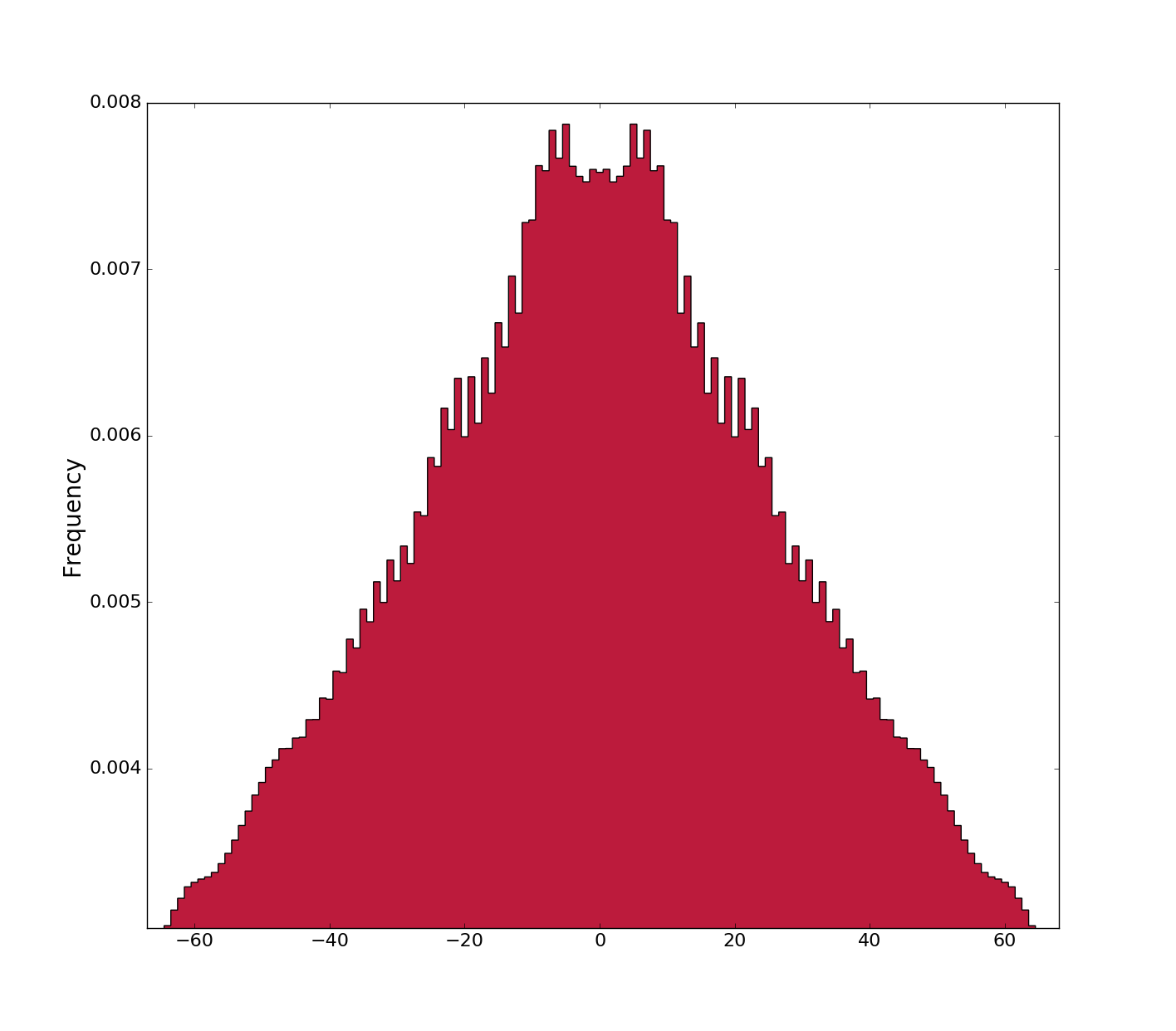}
		\caption{$\mathcal D_k^2$, close-up}
		\label{fig:Hist_DD2_1e12_zoom}
	\end{subfigure}
	\caption{Close-up of $\mathcal D_k^1$ and $\mathcal D_k^2$, showing the local maxima.}
	\label{fig:Histogram_DD2}
\end{figure}

\begin{claim}\label{claim:D1not35}
	$\mathcal D^1$ takes arbitrarily large values. However, $\mathcal D^1$ never takes the values 3 or 5.
\end{claim}
\begin{proof}
	The first assertion simply follows from the fact that $L_{\infty}(p_{k+1})-L_{\infty}(p_k)\geq p_{k+1}-p_k=D^1_k$ and $D^1_k$ takes arbitrarily large values. 
	
	For the second assertion we can assume that $p_k\neq 2$ (since $\mathcal D_1^1=1$). Since $\mathcal D_k^1\geq p_{k+1}-p_k$, $\mathcal D_k^1$ could equal 3 if and only if $p_{k+1}=p_k+2$ and similarly $\mathcal D_k^1$ could equal 5 if and only if $p_{k+1}=p_k+2$ or $p_{k+1}=p_k+4$.
	
	If $p_{k+1}=p_k+2$, then by definition \eqref{def:Linfty}:
	\begin{equation}\label{eq:D^1twinprime}
	\mathcal D_k^1= \max\{\|p_k\|_\infty, \|p_k+1\|_\infty\} + \max\{\|p_k+1\|_\infty, \|p_k+2\|_\infty\} = 2 \|p_k+1\|_\infty,
	\end{equation}
	which is clearly an even number.
	
	If $p_{k+1}=p_k+4$, then either $p_k+1$ or $p_k+3$ is divisible by 4, hence has $\ell_\infty$ norm at least $2$ and so $\mathcal D_k^1\geq 2+2+1+1 >5$.
\end{proof}

Polignac's Conjecture for $\mathcal D^1_k$ can be formulated as follows.
\begin{conjecture}[Modified Polignac's Conjecture]
	For any $N\in\mathbb{N}\setminus\{1,3,5\}$ there are infinitely many $k$ such that $\mathcal D^1_k=N$.
\end{conjecture}
It is also not difficult to characterize $\mathcal D^1$ for the first few small values.
\begin{claim}
	$\mathcal D_k^1$  equals (a) 2, (b) 4, or (c) 6 for some $k$ if and only if
	\begin{enumerate}[(a)]
		\item $p_{k+1}$ and $p_k$ are twin prime and $p_k+1$ is a square-free number,
		\item $p_{k+1}$ and $p_k$ are twin prime and $\|p_k+1\|_\infty=2$,
		\item $p_{k+1}$ and $p_k$ are twin prime and $\|p_k+1\|_\infty=3$ $\;$OR$\;$  $p_{k+1}=p_k+4$ and $\|p_k+2\|_\infty=1$.
	\end{enumerate}
	In general, if $p_{k+1}$ and $p_k$ are twin prime and $\|p_k+1\|_\infty=n$ then $\mathcal D_k^1=2n$.
\end{claim}
\begin{proof}
	We saw at the end of the proof of Claim \ref{claim:D1not35} that if  $p_{k+1}-p_k\geq4$ then $\mathcal D_k^1>5$. Hence, if $\mathcal D_k^1=2$ or $4$, then $p_{k+1}$ and $p_k$ must be twin prime. Moreover, \eqref{eq:D^1twinprime} implies that $\|p_k+1\|_\infty=1$ or $2$, respectively. If $\mathcal D_k^1=6$ then either $p_{k+1}$ and $p_k$ are twin prime and $\|p_k+1\|_\infty=3$ or $p_{k+1}=p_k+4$ ($p_{k+1}\geq p_k+6$ not possible). If $p_{k+1}=p_k+4$ then it further follows that one of $p_k+1$ or $p_k+3$ has $\ell_{\infty}$ norm equal to $2$ and so $p_k+2$ must be square-free.
	
	The other directions and the last claim are just trivial calculations.
\end{proof}

Another key difference between $D^1$ and $\mathcal D^1$ is that $\mathcal D^1$ makes a distinction between twin primes, according to the norm of the number between them. We compared the relative frequency amongst the first K twin primes $(p_k,p_{k+1})$ of the norm $\|p_k+1\|_{\infty}=\ell$ to the limiting distribution of the norm of a number between two square-free numbers given by $\lim_{n\to\infty}\Pi_n(1,\ell,1)/\sum_k\Pi_n(1,k,1)$ from Proposition~\ref{prop:31}.

Table~\ref{table:twinprime} shows that the relative frequencies of $\|p_k+1\|_{\infty}=\ell$ settle around specific constants which are not equal to the ones given by Proposition~\ref{prop:31}. Most notably, the relative frequencies are not strictly decreasing, the relative frequency of $\|p_k+1\|_{\infty}=2$ is larger than the one for $\|p_k+1\|_{\infty}=1$. A first step toward determining these constants would be to study the limiting distribution of $\|p_k+1\|_{\infty}=\ell$ along the subsequence of primes.

\begin{problem}
	Let $P_n$ denote a random prime number chosen uniformly at random from the set $\{1,2,\ldots,n\}$. Does $\|P_n-1\|_{\infty}$ converge in distribution? If so, what is the distribution of the limiting random variable?
\end{problem}

\begin{table}[h]
	\caption{Comparing $\|N\|_{\infty}=\ell$ when $\|N\pm1\|_{\infty}=1$ (first column) to the relative frequencies $\|p_k+1\|_{\infty}=\ell$ amongst the first K twin primes $(p_k,p_{k+1})$.} \label{table:twinprime}
	\begin{center}
		\begin{tabular}{c|r|llll}
			\toprule
			$\ell$ & if $\|N\pm1\|_{\infty}=1$ & $K=1\times10^8$ & $K=2\times10^8$ & $K=4\times10^8$ & $K=8\times10^8$ \\ \midrule
			\rowcolor[HTML]{E8E8E8}
			$1$ & 0.3889570915 & 0.29185522
			& 0.29183928
			& 0.291823840
			& 0.29180058625
			\\
			$2$ & 0.3156950811
			& 0.36190949
			& 0.36189782
			& 0.361889505
			& 0.361905095
			\\
			\rowcolor[HTML]{E8E8E8}
			$3$ & 0.1545003396
			& 0.18599846
			& 0.18601558
			& 0.1860339925
			& 0.186028890
			\\
			$4$ & 0.0730237294
			& 0.08559793
			& 0.085607245
			& 0.085600065
			& 0.08560440375
			\\
			\rowcolor[HTML]{E8E8E8}
			$5$ & 0.0348180734
			& 0.03929783
			& 0.039304695
			& 0.0393018075
			& 0.039307170
			\\
			$6$ & 0.0168170300
			& 0.01835525
			& 0.01835070
			& 0.0183527025
			& 0.01835503625
			\\
			\bottomrule
		\end{tabular}
	\end{center}
\end{table}

\appendix

\section{Sagemath Python code used to calculate \texorpdfstring{$L_{\infty}(N)$}{L_infty(N)}}\label{sec:Code}

Computations were carried out with \textsl{Sagemath} 7.2 \cite{SageMath} running on 8 hyperthreaded 3.6 GHz i7 CPU cores, and using 32 GB of RAM and about 500 GB of harddisk storage on a 64-bit Linux desktop computer. It took about 80 days of total walltime to complete all computations up to the order of $N=10^{12}$. Hardware resources were maximized by running the computations in 20 independent batches and utilizing a hand-tuned, load-balanced parallelization protocol. The code listing below is mostly self explanatory with comments and we provide detailed explanations where necessary, referring to line numbers.

(\ref{bignumpy}) The only non-standard package we used is called \textsl{bignumpy} \cite{Bignumpy} and it is crucial for seamlessly handling very large \textsl{numpy} arrays. Our computations heavily utilize \textsl{numpy} array operations but these become prohibitive using standard \textsl{numpy} when the size of a \textsl{numpy} array exhausts available memory. \textsl{Bignumpy} provides file backed \textsl{numpy} array objects for \textsl{Python} utilizing the mmap/shared memory feature of Unix. \textsl{Bignumpy} allows for streamlined array operations on arrays multiple times larger than available memory, with a negligible slowdown compared to system swap.

(\ref{parallel}) The \textsl{@parallel} decoration applied to the \textsl{calc\_norm00} function invokes the parallel interface, which means that the single set of $(N, M)$ function arguments will be replaced by a list of multiple $(N1, M1), (N2, M2),..., (Nk, Mk)$ arguments, see (\ref{defcalcnorm00}) and (\ref{arglist}). The parallel interface will then execute, simultaneously, multiple copies of \textsl{calc\_norm00} running on multiple ($ncpus=8$) CPU cores in parallel. The argument list is automatically divided up and distributed among the parallel processes, and in the end, the results are reassembled in a single list.

(\ref{inmost}) This is the innermost loop---the rate limiting computation of integer factorization. The  \textsl{factor} function (\ref{factor}) is a wrapper around the standard \textsl{PARI} factorization routine \cite{PARIlibrary}. Note that we did not use specialized factorization algorithms suitable for certain classes of numbers, because the computation of  $L_\infty (N)$ requires the factorization of every single number $1-N$. The number of iterations in the innermost loop should be hand tuned to achive optimum balance between computation and data manipulation (see more about this below). Also note that determining whether or not a number $N$ is prime can, of course, be calculated much faster than full factorization but \textsl{factor} has already been called on $N$.

(\ref{input_beg}-\ref{input_end})  The input parameters $N\_*$ allow flexibility to compute $L_\infty (N)$ depending on the hardware configuration. $N\_min$ and $N\_max$ define a particular segment, in our computations we used 20 segments each $5*10^{10}$ long to get to $10^{12}$. The workflow is organized in three nested loops. We already mentioned the innermost loop (\ref{inmost}), which includes multiple factorizations calculated in a single call to \textsl{calc\_norm00}. The inner loop (\ref{inner}) iterates multiple calls to \textsl{calc\_norm00} using the parallel interface as explained above (\ref{parallel}). Note that the argument list (\ref{arglist}) is a standard \textsl{Python} list, which has a significant memory footprint and, therefore, the balance between the chosen values of $N\_inmost\_loop$ and $N\_inner\_loop$ is crucial. i) Their product $N\_maxmem$ should be set such that the resulting \textsl{Python} list (\ref{arglist}) fits comfortably in memory. The associated work arrays (\ref{array_beg}-\ref{array_end}) are all \textsl{bignumpy} arrays and their size is not limiting. ii) $N\_inmost\_loop$ should be set $>>1$ to make sure that computation (factorization) dominates because the overhead of data manipulation associated with a single function call to \textsl{calc\_norm00} in the parallel environment is significant. In fact, the maximum overall speedup using 8 CPU cores was limited to less than five fold.

(\ref{L00_array}) The core data set at the heart of our computations is comprised of the $L_\infty (N)$ values at prime ``stops'' along the number trail. The associated \textsl{bignumpy} array can be generated piece wise by concatenating consecuitve sub-arrays generated in a succession of batch calculations as noted above. Since the file associated with a \textsl{bignumpy} array is the exact binary copy of the array's memory image, these files can readily be concatenated using the UNIX \textsl{cat} command.

(\ref{hopseq}, \ref{cumsum}, \ref{prn_hopseq}, \ref{prn_cumsum})  $L_\infty (N)$ is computed as the cumulative sum of the sequence of ``hops'' between consecutive numbers along the number trail and, therefore, every new batch computation needs two data points from the previous calculation to start with. One is the length of the last hop and the other is the current value of $L_\infty (N)$.

(\ref{outer_beg}-\ref{outer_end}) The cumulative summation is carried out in the outer loop utilizing a number of \textsl{numpy} operations. i) First, the sequence of hops is computed for a continuous segment $N_{beg}-N_{end}$ by taking the maximum of every two adjacent infinity norm values (Chebyshev contour indices), and the values are stored in \textsl{Hop\_Sequence\_Arr} (\ref{np_max}). ii) The cumulative sum is then computed in two steps (\ref{cumsum2}, \ref{+=cumsum}). (The current values of \textsl{last\_hop\_seq} and \textsl{cumsum} are saved for the next segment (\ref{hopseq2}, \ref{hopseq3}, \ref{cumsum3}), see previous paragraph.) iii) Finally, \textsl{Prime\_bIndex\_Arr} is utilized as a binary mask (\ref{mask}) to keep only the prime stops and store them in \textsl{Prime\_Stops\_onL00\_Arr} (\ref{prime_stops}).

(\ref{arglist}) Note that the output list generated from the return values of the parallel calculation (\ref{ret_lst}) is not guaranteed to preserve the order of the input list, and must be sorted.

(\ref{beg_cumsum}-\ref{end_cumsum}) Listing of the \textsl{last\_hop\_seq} and \textsl{cumsum} values printed at lines \ref{prn_hopseq}, \ref{prn_cumsum} after completion of each of the 20 segments of the master calculation.

(\ref{beg_diffs}-\ref{end_diffs}) This section of the code reads in the concatenated file \textsl{Prime\_Stops\_onL00\_\_N=1-1000000000000.mmap} holding the $L_\infty(p)$ values at the prime ``stops'' and computes the $\mathcal D^1$ and $\mathcal D^2$ differences.

(\ref{histplot}) The final section computes and plots the $\mathcal D^1$ and $\mathcal D^2$ histograms.

\lstset{
	language=Python,
	showstringspaces=false,
	keywordstyle=\color{black}\bfseries,
	commentstyle=\itshape,
	basicstyle={\ttfamily,\small},
	stringstyle=\slshape,
	numbers=left, numberstyle=\tiny,
	emph={find, calc_norm00, factor,prime_pi},
	emphstyle=\underline,
	breaklines,
	breakatwhitespace,
	breakindent=20pt,
	escapeinside={(*@}{@*)}
}
\begin{lstlisting}
import numpy as np
from bignumpy import bignumpy(*@\label{bignumpy}@*)
####################################################################

def find(name, path):
for root, dirs, files in os.walk(path):
if name in files:
return os.path.join(root, name)
####################################################################

# Calc for N: Chebyshev contour index (norm00)
@parallel(p_iter='multiprocessing', ncpus=8)(*@\label{parallel}@*)
def calc_norm00(N,M):(*@\label{defcalcnorm00}@*)

N += 1                # N is passed as a zero based array index
norm00     = []
is_a_prime = []
if N==1:
norm00     = [0]
is_a_prime = [0]

for i in xrange( max( 2,N),N+M):               # innermost loop(*@\label{inmost}@*)

f = factor( i)(*@\label{factor}@*)
tuples = [x for x   in f]
#primes = [p for p,e in f]
powers = [e for p,e in f]
norm00.append( max( powers))               # infinity norm
if len( tuples)==1 and tuples[0][1]==1:    # 'i' is a prime
is_a_prime.append( 1)
else:
is_a_prime.append( 0)

return norm00, is_a_prime(*@\label{ret_lst}@*)
####################################################################

%timeit
N_min         = 450000000001(*@\label{input_beg}@*)
N_max         = 500000000000
N_primes      = prime_pi( N_max) - prime_pi( N_min-1)
N_outer_loop  = 5*10^3
N_inner_loop  = 10^4
N_inmost_loop = 10^3
N_maxmem      = N_inner_loop * N_inmost_loop          # max array size stored in memory(*@\label{input_end}@*)
ext           = str(N_min)+"-"+str( N_max)+".mmap"    # file extension

# Chebyshev contour index array --create(*@\label{array_beg}@*)
Contour_Indx_mmap_fname = "Contour_Indx__N="+ext
Contour_Indx_Arr = bignumpy( Contour_Indx_mmap_fname, np.uint8, (N_maxmem,))

# Is it a prime? Binary index array --create
Prime_bIndex_mmap_fname = "Prime_bIndex__N="+ext
Prime_bIndex_Arr = bignumpy( Prime_bIndex_mmap_fname, np.uint8,  N_maxmem,))

# Hop sequence array --create
Hop_Sequence_mmap_fname = "Hop_Sequence__N="+ext
Hop_Sequence_Arr = bignumpy( Hop_Sequence_mmap_fname, np.uint8, (N_maxmem,))

# Temp array for data type casting --create(*@\label{array_end}@*)
Temp1_Arr = bignumpy( 'Temp1.mmap', np.int64, (N_maxmem,))
Temp2_Arr = bignumpy( 'Temp2.mmap', np.int64, (N_maxmem,))

# L00 values at prime stops --create(*@\label{L00_array}@*)
Prime_Stops_onL00_mmap_fname = "Prime_Stops_onL00__N="+ext
Prime_Stops_onL00_Arr = bignumpy( Prime_Stops_onL00_mmap_fname, np.int64, (N_primes,))

incr = N_inner_loop * N_inmost_loop
stp  = N_inmost_loop
last_prime_stop = 0
last_hop_seq    = 11(*@\label{hopseq}@*)
cumsum          = 1029766280643(*@\label{cumsum}@*)
tim  = walltime()
for i in xrange( N_outer_loop):                           # OUTER LOOP(*@\label{outer_beg}@*)

beg        = i *   incr + N_min-1     # array indx
end        = beg + incr
input_list = [(x,stp) for x in range(beg,end,stp)]    # inner loop(*@\label{inner}@*)
tuples = sorted( list( calc_norm00( input_list)))(*@\label{arglist}@*)

output_list = [ x[1][0] for x in tuples ]
flat_list = [item for sublist in output_list for item in sublist]
Contour_Indx_Arr = np.fromiter( flat_list, np.uint8)
if beg==0:
Hop_Sequence_Arr[0] = 0
else:
Hop_Sequence_Arr[0] = max( Contour_Indx_Arr[0], last_hop_seq)(*@\label{hopseq2}@*)
last_hop_seq = Contour_Indx_Arr[-1](*@\label{hopseq3}@*)
np.maximum( Contour_Indx_Arr[1:], Contour_Indx_Arr[:-1], out=Hop_Sequence_Arr[1:])(*@\label{np_max}@*)

output_list = [ x[1][1] for x in tuples ]
flat_list   = [item for sublist in output_list for item in sublist]
Prime_bIndex_Arr = np.fromiter( flat_list, np.uint8)

Temp1_Arr[:] = Hop_Sequence_Arr
Temp1_Arr.cumsum( out=Temp2_Arr)(*@\label{cumsum2}@*)
Temp2_Arr += cumsum(*@\label{+=cumsum}@*)
cumsum = Temp2_Arr[-1](*@\label{cumsum3}@*)
np.multiply( Prime_bIndex_Arr, Temp2_Arr, out=Temp1_Arr)  # binary mask(*@\label{mask}@*)
count_nonzero = np.count_nonzero( Temp1_Arr)
Prime_Stops_onL00_Arr[last_prime_stop:last_prime_stop+count_nonzero] = Temp1_Arr[np.nonzero( Temp1_Arr)](*@\label{prime_stops}@*)
last_prime_stop += count_nonzero(*@\label{outer_end}@*)

if( (i+1) % (N_outer_loop // 100)) == 0:
print str(int(100*float(i+1)/float(N_outer_loop))).rjust(3)+'%', 'Wall time = ', str(walltime(tim)).rjust(20)


#print Contour_Indx_Arr[...]
#print Prime_bIndex_Arr[...]
#print Hop_Sequence_Arr[...]

#print Prime_Stops_onL00_Arr[...]

print "last_hop_seq= ", last_hop_seq(*@\label{prn_hopseq}@*)
print "cumsum=       ", cumsum(*@\label{prn_cumsum}@*)
####################################################################

#                             last_hop_seq    cumsum(*@\label{beg_cumsum}@*)
#            1-  50000000000  11              114418475903
#  50000000001- 100000000000  11              228836951528
# 100000000001- 150000000000  11              343255427228
# 150000000001- 200000000000  12              457673901868
# 200000000001- 250000000000  12              572092378372
# 250000000001- 300000000000  11              686510853733
# 300000000001- 350000000000  11              800929329310
# 350000000001- 400000000000  13              915347805373
# 400000000001- 450000000000  11             1029766280643
# 450000000001- 500000000000  12             1144184756086
# 500000000001- 550000000000  11             1258603231525
# 550000000001- 600000000000  12             1373021707238
# 600000000001- 650000000000  11             1487440183820
# 650000000001- 700000000000  11             1601858659108
# 700000000001- 750000000000  12             1716277134939
# 750000000001- 800000000000  14             1830695610669
# 800000000001- 850000000000  11             1945114085159
# 850000000001- 900000000000  11             2059532560893
# 900000000001- 950000000000  11             2173951036235
# 950000000001-1000000000000  12             2288369511216(*@\label{end_cumsum}@*)
####################################################################


N_max = 10^12
N_primes = prime_pi( N_max)(*@\label{beg_diffs}@*)
ext = str( N_max)+".mmap"
dir = "/where/to/find/my_files/"

Prime_Stops_onL00_mmap_fname = "Prime_Stops_onL00__N=1-"+ext
Prime_Stops_onL00_Arr = bignumpy( find( Prime_Stops_onL00_mmap_fname, dir), np.int64, (N_primes,))        # L00 at prime stops --read

Prime_diff1_onL00_mmap_fname = "Prime_diff1_onL00__N=1-"+ext
Prime_diff1_onL00_Arr = bignumpy( Prime_diff1_onL00_mmap_fname, np.int16, ( N_primes-1,))                 # --create
np.subtract( Prime_Stops_onL00_Arr[1:], Prime_Stops_onL00_Arr[:-1], out=Prime_diff1_onL00_Arr)

Prime_diff2_onL00_mmap_fname = "Prime_diff2_onL00__N=1-"+ext
Prime_diff2_onL00_Arr = bignumpy( Prime_diff2_onL00_mmap_fname, np.int16, ( N_primes-2,))                 # --create
np.subtract( Prime_diff1_onL00_Arr[1:], Prime_diff1_onL00_Arr[:-1], out=Prime_diff2_onL00_Arr)(*@\label{end_diffs}@*)
####################################################################

# Plot data

import matplotlib.pyplot as plt(*@\label{histplot}@*)

E_max    = 12
N_max    = 10^12
N_primes = prime_pi( N_max)

ext = str(N_max)+".mmap"
dir = "/where/to/find/my_files/"

Prime_diff1_onL00_mmap_fname = "Prime_diff1_onL00__N=1-"+ext
Prime_diff1_onL00_Arr = bignumpy( find( Prime_diff1_onL00_mmap_fname, dir), np.int16, ( N_primes-1,))                # --read

Prime_diff2_onL00_mmap_fname = "Prime_diff2_onL00__N=1-"+ext
Prime_diff2_onL00_Arr = bignumpy( find( Prime_diff2_onL00_mmap_fname, dir), np.int16, ( N_primes-2,))                # --read

slice_diff1 = [prime_pi( 10^x)-1 for x in range( E_max+1)]
slice_diff2 = [prime_pi( 10^x)-2 for x in range( E_max+1)]

bin_edges_Prime_diff1_onL00 = np.arange( np.amin( Prime_diff1_onL00_Arr),np.amax( Prime_diff1_onL00_Arr)+2)            # '+2' to make sure rightmost bin is empty
bin_edges_Prime_diff2_onL00 = np.arange( np.amin( Prime_diff2_onL00_Arr),np.amax( Prime_diff2_onL00_Arr)+2)

# Histogram raw data:

slice_indx = 12  # >= 1
sdiff1 = slice_diff1[slice_indx]
sdiff2 = slice_diff2[slice_indx]

hist, bin_edges = np.histogram( Prime_diff1_onL00_Arr[:sdiff1], bins=bin_edges_Prime_diff1_onL00)
start = np.amin( np.nonzero(hist))
end   = np.amax( np.nonzero(hist))
np.savetxt( 'Prime_diff1_onL00_histogram__N=1-'+str( 10^slice_indx)+'.txt', hist[start:end+1],      fmt='%12d')
np.savetxt( 'Prime_diff1_onL00_bin_edges__N=1-'+str( 10^slice_indx)+'.txt', bin_edges[start:end+1], fmt='%12d')

hist, bin_edges = np.histogram( Prime_diff2_onL00_Arr[:sdiff2], bins=bin_edges_Prime_diff2_onL00)
start = np.amin( np.nonzero( hist))
end   = np.amax( np.nonzero( hist))
np.savetxt( 'Prime_diff2_onL00_histogram__N=1-'+str(10^slice_indx)+'.txt', hist[start:end+1],      fmt='%12d')
np.savetxt( 'Prime_diff2_onL00_bin_edges__N=1-'+str(10^slice_indx)+'.txt', bin_edges[start:end+1], fmt='%12d')

# Individual plots:

slice_indx = 12  # >= 1
my_color = (188/255, 27/255, 60/255)

plt.figure(1)
plt.hist( Prime_diff1_onL00_Arr[:slice_diff1[slice_indx]], bins=bin_edges_Prime_diff1_onL00, color=my_color, normed=True, align='left', histtype='stepfilled')
plt.ylabel( 'Frequency', fontsize=20)
plt.yticks( fontsize=16)
plt.xticks( fontsize=16)

plt.figure(2)
plt.hist( Prime_diff2_onL00_Arr[:slice_diff2[slice_indx]], bins=bin_edges_Prime_diff2_onL00, color=my_color, normed=True, align='left', histtype='stepfilled')
plt.ylabel( 'Frequency', fontsize=20)
plt.yticks( fontsize=16)
plt.xticks( fontsize=16)

# Multiple plots:

my_colors = [(0/255, 153/255, 153/255), (252/255, 126/255, 0/255), (251/255, 2/255, 1/255), (102/255, 205/255, 204/255), (255/255, 175/255, 103/255), (255/255, 102/255, 102/255), (0/255, 103/255, 102/255), (178/255, 86/255, 1/255), (175/255, 1/255, 2/255), (83/255, 165/255, 161/255), (246/255, 113/255, 78/255)]

plt.figure(3)
plt.hist( (Prime_diff1_onL00_Arr[:slice_diff1[2]], Prime_diff1_onL00_Arr[:slice_diff1[3]], Prime_diff1_onL00_Arr[:slice_diff1[4]], Prime_diff1_onL00_Arr[:slice_diff1[5]], Prime_diff1_onL00_Arr[:slice_diff1[6]], Prime_diff1_onL00_Arr[:slice_diff1[7]], Prime_diff1_onL00_Arr[:slice_diff1[8]], Prime_diff1_onL00_Arr[:slice_diff1[9]], Prime_diff1_onL00_Arr[:slice_diff1[10]], Prime_diff1_onL00_Arr[:slice_diff1[11]], Prime_diff1_onL00_Arr), bins=bin_edges_Prime_diff1_onL00, color=my_colors, normed=False, align='left', log=True, stacked=True, histtype='stepfilled')
plt.ylabel( 'log Frequency', fontsize=20)
plt.yticks( fontsize=16)
plt.xticks( fontsize=16)

plt.figure(4)
plt.hist( (Prime_diff2_onL00_Arr[:slice_diff2[2]], Prime_diff2_onL00_Arr[:slice_diff2[3]], Prime_diff2_onL00_Arr[:slice_diff2[4]], Prime_diff2_onL00_Arr[:slice_diff2[5]], Prime_diff2_onL00_Arr[:slice_diff2[6]], Prime_diff2_onL00_Arr[:slice_diff2[7]], Prime_diff2_onL00_Arr[:slice_diff2[8]], Prime_diff2_onL00_Arr[:slice_diff2[9]], Prime_diff2_onL00_Arr[:slice_diff2[10]], Prime_diff2_onL00_Arr[:slice_diff2[11]], Prime_diff2_onL00_Arr), bins=bin_edges_Prime_diff2_onL00, color=my_colors, normed=False, align='left', log=True, stacked=True, histtype='stepfilled')
plt.ylabel( 'log Frequency', fontsize=20)
plt.yticks( fontsize=16)
plt.xticks( fontsize=16)

plt.show()
####################################################################
\end{lstlisting}

\section{Tabulated histograms of prime gaps}\label{sec:data}
\begin{footnotesize}
	\begin{center}
		\begin{longtable}{r|rrrrrrrrrrr}
			\captionsetup{justification=centering}
			\caption{Numerical data of  histograms of  $\mathcal D^1$ differences in column one for $N\leq10^{2-12}$ tabulated in columns 2-12}\label{tab:Hist_DD1_1e2-12}\\
			\toprule
			\multicolumn{1}{c|}{$\mathcal D^1$}  & \multicolumn{1}{c}{$10^{12}$}  & \multicolumn{1}{c}{$10^{11}$}  & \multicolumn{1}{c}{$10^{10}$}  & \multicolumn{1}{c}{$10^{9}$}  & \multicolumn{1}{c}{$10^{8}$}  & \multicolumn{1}{c}{$10^{7}$}  & \multicolumn{1}{c}{$10^{6}$}  & \multicolumn{1}{c}{$10^{5}$}  & \multicolumn{1}{c}{$10^{4}$}  & \multicolumn{1}{c}{$10^{3}$}  & \multicolumn{1}{c}{$10^{2}$}  \\ \midrule
			\endfirsthead
			\multicolumn{12}{c}{\tablename\ \thetable\ -- \textit{Continued from previous page}} \\ \toprule
			\multicolumn{1}{c|}{$\mathcal D^1$}  & \multicolumn{1}{c}{$10^{12}$}  & \multicolumn{1}{c}{$10^{11}$}  & \multicolumn{1}{c}{$10^{10}$}  & \multicolumn{1}{c}{$10^{9}$}  & \multicolumn{1}{c}{$10^{8}$}  & \multicolumn{1}{c}{$10^{7}$}  & \multicolumn{1}{c}{$10^{6}$}  & \multicolumn{1}{c}{$10^{5}$}  & \multicolumn{1}{c}{$10^{4}$}  & \multicolumn{1}{c}{$10^{3}$}  & \multicolumn{1}{c}{$10^{2}$}  \\ \midrule
			\endhead
			\multicolumn{12}{r}{\textit{\small{Continued on next page}}} \\
			\endfoot
			\endlastfoot
			\rowcolor[HTML]{E8E8E8} 
			1  & 1         & 1        & 1        & 1       & 1      & 1     & 1    & 1   & 1  & 1  & 1 \\
			2  & 545836959 & 65482891 & 8000621  & 999952  & 128678 & 17238 & 2352 & 358 & 63 & 12 & 3 \\
			\rowcolor[HTML]{E8E8E8} 
			3  & 0         & 0        & 0        & 0       & 0      & 0     & 0    & 0   & 0  & 0  & 0 \\
			4  & 676962502 & 81197873 & 9918050  & 1238662 & 159572 & 21273 & 2962 & 460 & 83 & 14 & 4 \\
			\rowcolor[HTML]{E8E8E8} 
			5  & 0         & 0        & 0        & 0       & 0      & 0     & 0    & 0   & 0  & 0  & 0 \\
			6  & 812569885 & 97460711 & 11910458 & 1489205 & 191644 & 25567 & 3508 & 519 & 86 & 17 & 3 \\
			\rowcolor[HTML]{E8E8E8} 
			7  & 233485333 & 28008641 & 3421613  & 427267  & 54555  & 7218  & 1024 & 151 & 24 & 3  & 1 \\
			8  & 675924191 & 80872118 & 9851261  & 1225160 & 156491 & 20793 & 2882 & 439 & 70 & 13 & 1 \\
			\rowcolor[HTML]{E8E8E8} 
			9  & 382285674 & 45657562 & 5551012  & 689857  & 88083  & 11685 & 1644 & 241 & 34 & 10 & 2 \\
			10 & 599704882 & 71643947 & 8701213  & 1078829 & 137483 & 18181 & 2481 & 353 & 52 & 9  & 1 \\
			\rowcolor[HTML]{E8E8E8} 
			11 & 275043047 & 32785324 & 3977086  & 492050  & 62471  & 8189  & 1152 & 163 & 30 & 0  & 0 \\
			12 & 756641841 & 90135360 & 10920595 & 1351869 & 171761 & 22303 & 3026 & 420 & 63 & 4  & 1 \\
			\rowcolor[HTML]{E8E8E8} 
			13 & 384294766 & 45758060 & 5540132  & 683669  & 86297  & 11231 & 1547 & 239 & 30 & 6  & 1 \\
			14 & 677173536 & 80473232 & 9718049  & 1196206 & 150599 & 19575 & 2670 & 369 & 63 & 9  & 1 \\
			\rowcolor[HTML]{E8E8E8} 
			15 & 401018460 & 47543338 & 5723867  & 703177  & 88375  & 11521 & 1513 & 237 & 37 & 5  & 0 \\
			16 & 620181299 & 73474151 & 8841576  & 1083654 & 135551 & 17349 & 2271 & 311 & 49 & 13 & 3 \\
			\rowcolor[HTML]{E8E8E8} 
			17 & 424850184 & 50195576 & 6020340  & 735419  & 91699  & 11709 & 1583 & 224 & 27 & 4  & 1 \\
			18 & 597556733 & 70592210 & 8465897  & 1033639 & 128604 & 16478 & 2171 & 293 & 39 & 3  & 0 \\
			\rowcolor[HTML]{E8E8E8} 
			19 & 417182685 & 49185886 & 5882046  & 715767  & 88879  & 11340 & 1491 & 199 & 22 & 6  & 0 \\
			20 & 500276741 & 58952639 & 7052759  & 857727  & 106288 & 13622 & 1822 & 248 & 35 & 8  & 1 \\
			\rowcolor[HTML]{E8E8E8} 
			21 & 407144642 & 47774244 & 5683277  & 686540  & 84352  & 10560 & 1395 & 181 & 27 & 5  &   \\
			22 & 484516599 & 56852120 & 6763219  & 816597  & 100701 & 12564 & 1623 & 198 & 28 & 3  &   \\
			\rowcolor[HTML]{E8E8E8} 
			23 & 443439269 & 51900616 & 6153871  & 740041  & 90225  & 11240 & 1434 & 199 & 19 & 1  &   \\
			24 & 470479351 & 55067706 & 6528288  & 784364  & 96024  & 12067 & 1551 & 198 & 26 & 1  &   \\
			\rowcolor[HTML]{E8E8E8} 
			25 & 486683648 & 56773666 & 6700619  & 802366  & 97486  & 11964 & 1418 & 158 & 18 & 2  &   \\
			26 & 492742414 & 57478966 & 6784519  & 812971  & 98737  & 12325 & 1526 & 196 & 30 & 3  &   \\
			\rowcolor[HTML]{E8E8E8} 
			27 & 447830032 & 52126930 & 6134631  & 731493  & 88573  & 11033 & 1263 & 142 & 19 & 1  &   \\
			28 & 415327048 & 48370523 & 5701776  & 680946  & 82323  & 10102 & 1263 & 165 & 19 & 2  &   \\
			\rowcolor[HTML]{E8E8E8} 
			29 & 390303895 & 45329992 & 5320982  & 630812  & 76393  & 9259  & 1153 & 141 & 14 & 1  &   \\
			30 & 392437767 & 45575573 & 5347835  & 635587  & 76130  & 9320  & 1167 & 133 & 21 & 3  &   \\
			\rowcolor[HTML]{E8E8E8} 
			31 & 411121720 & 47566224 & 5555827  & 655573  & 78270  & 9417  & 1109 & 145 & 15 & 0  &   \\
			32 & 391706876 & 45315055 & 5292173  & 625088  & 74690  & 8998  & 1063 & 119 & 8  & 2  &   \\
			\rowcolor[HTML]{E8E8E8} 
			33 & 407916113 & 47009979 & 5467687  & 641287  & 75646  & 8971  & 1057 & 117 & 10 & 1  &   \\
			34 & 357435898 & 41178422 & 4788218  & 561466  & 66411  & 7902  & 969  & 127 & 15 & 0  &   \\
			\rowcolor[HTML]{E8E8E8} 
			35 & 415666860 & 47710551 & 5519737  & 643360  & 75712  & 9030  & 1089 & 120 & 9  & 2  &   \\
			36 & 357761163 & 41026033 & 4741553  & 552125  & 64897  & 7588  & 901  & 95  & 6  & 1  &   \\
			\rowcolor[HTML]{E8E8E8} 
			37 & 438008056 & 50053582 & 5759940  & 666505  & 77507  & 8914  & 1032 & 109 & 7  & 0  &   \\
			38 & 362734703 & 41403690 & 4757877  & 548943  & 63884  & 7324  & 856  & 110 & 11 & 0  &   \\
			\rowcolor[HTML]{E8E8E8} 
			39 & 419338332 & 47753944 & 5473818  & 631038  & 72661  & 8342  & 906  & 88  & 12 & 0  &   \\
			40 & 319372932 & 36349710 & 4161936  & 478868  & 55499  & 6364  & 687  & 72  & 9  & 1  &   \\
			\rowcolor[HTML]{E8E8E8} 
			41 & 384186364 & 43587533 & 4972936  & 570258  & 65417  & 7486  & 844  & 92  & 6  & 0  &   \\
			42 & 310770326 & 35210712 & 4008763  & 458211  & 52511  & 5895  & 656  & 69  & 6  & 0  &   \\
			\rowcolor[HTML]{E8E8E8} 
			43 & 375151082 & 42408631 & 4809769  & 548333  & 62368  & 7088  & 800  & 83  & 6  & 0  &   \\
			44 & 301328314 & 34025495 & 3854173  & 438680  & 50067  & 5531  & 665  & 72  & 6  & 0  &   \\
			\rowcolor[HTML]{E8E8E8} 
			45 & 342858056 & 38657690 & 4375235  & 495639  & 55826  & 6342  & 710  & 74  & 7  & 0  &   \\
			46 & 274048457 & 30861067 & 3483910  & 394841  & 44666  & 4925  & 534  & 58  & 4  & 0  &   \\
			\rowcolor[HTML]{E8E8E8} 
			47 & 314323046 & 35331031 & 3980879  & 448389  & 50260  & 5599  & 625  & 65  & 5  & 0  &   \\
			48 & 268461025 & 30128434 & 3390308  & 381310  & 42656  & 4740  & 519  & 50  & 2  & 0  &   \\
			\rowcolor[HTML]{E8E8E8} 
			49 & 313748331 & 35145253 & 3944099  & 442484  & 49301  & 5395  & 602  & 67  & 4  & 1  &   \\
			50 & 279059340 & 31197050 & 3493274  & 390264  & 43288  & 4677  & 502  & 43  & 3  &    &   \\
			\rowcolor[HTML]{E8E8E8} 
			51 & 316192097 & 35270467 & 3938108  & 438494  & 48632  & 5216  & 555  & 54  & 2  &    &   \\
			52 & 290982321 & 32393863 & 3612103  & 401777  & 44226  & 4687  & 490  & 39  & 4  &    &   \\
			\rowcolor[HTML]{E8E8E8} 
			53 & 321173394 & 35704308 & 3971568  & 440735  & 48442  & 5238  & 546  & 46  & 4  &    &   \\
			54 & 296555755 & 32916809 & 3654866  & 404294  & 44375  & 4865  & 487  & 50  & 1  &    &   \\
			\rowcolor[HTML]{E8E8E8} 
			55 & 310144069 & 34377195 & 3805054  & 419457  & 46030  & 4780  & 485  & 38  & 4  &    &   \\
			56 & 278244969 & 30786869 & 3402596  & 375360  & 40765  & 4250  & 377  & 23  & 1  &    &   \\
			\rowcolor[HTML]{E8E8E8} 
			57 & 274938973 & 30365686 & 3349189  & 367275  & 39855  & 4219  & 452  & 36  & 5  &    &   \\
			58 & 251238297 & 27693610 & 3046692  & 332896  & 35716  & 3799  & 390  & 23  & 3  &    &   \\
			\rowcolor[HTML]{E8E8E8} 
			59 & 255057585 & 28064094 & 3078949  & 334786  & 35881  & 3707  & 364  & 26  & 1  &    &   \\
			60 & 240625481 & 26406855 & 2893001  & 314572  & 33529  & 3453  & 356  & 27  & 4  &    &   \\
			\rowcolor[HTML]{E8E8E8} 
			61 & 240098350 & 26312891 & 2874685  & 311734  & 33303  & 3428  & 317  & 25  & 0  &    &   \\
			62 & 230501322 & 25221244 & 2747421  & 296466  & 31590  & 3203  & 317  & 32  & 2  &    &   \\
			\rowcolor[HTML]{E8E8E8} 
			63 & 240699750 & 26287308 & 2858713  & 308409  & 32549  & 3323  & 281  & 17  & 1  &    &   \\
			64 & 245034155 & 26679075 & 2893017  & 309735  & 32703  & 3363  & 310  & 34  & 1  &    &   \\
			\rowcolor[HTML]{E8E8E8} 
			65 & 261624047 & 28405280 & 3071254  & 328501  & 34295  & 3350  & 309  & 32  & 3  &    &   \\
			66 & 266059296 & 28835866 & 3106800  & 331091  & 34325  & 3517  & 320  & 23  & 1  &    &   \\
			\rowcolor[HTML]{E8E8E8} 
			67 & 276300134 & 29870623 & 3211217  & 341332  & 35612  & 3414  & 301  & 24  & 1  &    &   \\
			68 & 271214452 & 29261959 & 3136401  & 331350  & 34448  & 3383  & 321  & 20  & 0  &    &   \\
			\rowcolor[HTML]{E8E8E8} 
			69 & 269383864 & 29009679 & 3101086  & 327085  & 33620  & 3299  & 272  & 13  & 0  &    &   \\
			70 & 254946493 & 27403933 & 2922126  & 306593  & 31439  & 3079  & 282  & 23  & 2  &    &   \\
			\rowcolor[HTML]{E8E8E8} 
			71 & 243206891 & 26090253 & 2773428  & 290862  & 29670  & 3004  & 241  & 21  & 1  &    &   \\
			72 & 225296489 & 24122260 & 2558040  & 267104  & 27016  & 2579  & 232  & 14  & 2  &    &   \\
			\rowcolor[HTML]{E8E8E8} 
			73 & 208369696 & 22276198 & 2361046  & 246199  & 24927  & 2467  & 223  & 20  & 0  &    &   \\
			74 & 196338136 & 20936263 & 2209352  & 229402  & 23131  & 2200  & 198  & 12  & 0  &    &   \\
			\rowcolor[HTML]{E8E8E8} 
			75 & 188391164 & 20046979 & 2110189  & 218518  & 21830  & 1943  & 156  & 10  & 0  &    &   \\
			76 & 186442645 & 19783796 & 2074672  & 213678  & 21235  & 1923  & 166  & 10  & 2  &    &   \\
			\rowcolor[HTML]{E8E8E8} 
			77 & 184682531 & 19558506 & 2046017  & 208819  & 20672  & 1945  & 179  & 12  & 1  &    &   \\
			78 & 183172563 & 19350290 & 2022416  & 207385  & 20373  & 1881  & 148  & 7   & 0  &    &   \\
			\rowcolor[HTML]{E8E8E8} 
			79 & 178996142 & 18878770 & 1964451  & 200764  & 19730  & 1845  & 135  & 9   & 0  &    &   \\
			80 & 172704652 & 18184358 & 1890852  & 192546  & 18867  & 1728  & 138  & 7   & 0  &    &   \\ \bottomrule
		\end{longtable}
	\end{center}
\end{footnotesize}

\begin{footnotesize}
	\begin{center}
		\begin{longtable}{r|rrrrrrrrrrr}
			\captionsetup{justification=centering}
			\caption{Numerical data of  histograms of  $\mathcal D^2$ differences in column one for $N\leq10^{2-12}$ tabulated in columns 2-12}\label{tab:Hist_DD2_1e2-12}\\
			\toprule
			\multicolumn{1}{c|}{$\mathcal D^2$}  & \multicolumn{1}{c}{$10^{12}$}  & \multicolumn{1}{c}{$10^{11}$}  & \multicolumn{1}{c}{$10^{10}$}  & \multicolumn{1}{c}{$10^{9}$}  & \multicolumn{1}{c}{$10^{8}$}  & \multicolumn{1}{c}{$10^{7}$}  & \multicolumn{1}{c}{$10^{6}$}  & \multicolumn{1}{c}{$10^{5}$}  & \multicolumn{1}{c}{$10^{4}$}  & \multicolumn{1}{c}{$10^{3}$}  & \multicolumn{1}{c}{$10^{2}$}  \\ \midrule
			\endfirsthead
			\multicolumn{12}{c}{\tablename\ \thetable\ -- \textit{Continued from previous page}} \\ \toprule
			\multicolumn{1}{c|}{$\mathcal D^2$}  & \multicolumn{1}{c}{$10^{12}$}  & \multicolumn{1}{c}{$10^{11}$}  &  \multicolumn{1}{c}{$10^{10}$}  & \multicolumn{1}{c}{$10^{9}$}  & \multicolumn{1}{c}{$10^{8}$}  & \multicolumn{1}{c}{$10^{7}$}  & \multicolumn{1}{c}{$10^{6}$}  & \multicolumn{1}{c}{$10^{5}$}  & \multicolumn{1}{c}{$10^{4}$}  & \multicolumn{1}{c}{$10^{3}$}  & \multicolumn{1}{c}{$10^{2}$}  \\ \midrule
			\endhead
			\multicolumn{12}{r}{\textit{\small{Continued on next page}}} \\
			\endfoot
			\endlastfoot
			\rowcolor[HTML]{E8E8E8} 
			-60 & 124678820 & 13617925 & 1479775 & 159214 & 16731 & 1719 & 140  & 9   & 2  &    &   \\
			-59 & 125417281 & 13698994 & 1493730 & 161701 & 17398 & 1755 & 162  & 20  & 3  &    &   \\
			\rowcolor[HTML]{E8E8E8} 
			-58 & 125911011 & 13764630 & 1500292 & 162002 & 17011 & 1657 & 172  & 18  & 0  &    &   \\
			-57 & 126923083 & 13894104 & 1516473 & 164037 & 17363 & 1785 & 182  & 18  & 0  &    &   \\
			\rowcolor[HTML]{E8E8E8} 
			-56 & 128904390 & 14133007 & 1547182 & 168320 & 18180 & 1859 & 172  & 20  & 1  &    &   \\
			-55 & 131234765 & 14416090 & 1578825 & 173076 & 18466 & 1934 & 185  & 14  & 0  &    &   \\
			\rowcolor[HTML]{E8E8E8} 
			-54 & 134222294 & 14782109 & 1626127 & 177965 & 19312 & 2032 & 193  & 16  & 1  &    &   \\
			-53 & 137524028 & 15185420 & 1676289 & 183678 & 19652 & 2153 & 231  & 15  & 1  &    &   \\
			\rowcolor[HTML]{E8E8E8} 
			-52 & 140779422 & 15571660 & 1718999 & 188928 & 20752 & 2206 & 201  & 11  & 1  &    &   \\
			-51 & 144402779 & 16017844 & 1776828 & 196358 & 21781 & 2337 & 214  & 21  & 4  &    &   \\
			\rowcolor[HTML]{E8E8E8} 
			-50 & 147259221 & 16369755 & 1820487 & 201629 & 21894 & 2394 & 240  & 14  & 0  &    &   \\
			-49 & 150647376 & 16780566 & 1870272 & 208691 & 23276 & 2460 & 258  & 21  & 1  &    &   \\
			\rowcolor[HTML]{E8E8E8} 
			-48 & 152336894 & 16989890 & 1895485 & 210985 & 23352 & 2546 & 266  & 22  & 2  &    &   \\
			-47 & 154896930 & 17306325 & 1936930 & 217237 & 24137 & 2622 & 272  & 30  & 1  &    &   \\
			\rowcolor[HTML]{E8E8E8} 
			-46 & 154945318 & 17317334 & 1937210 & 216378 & 23987 & 2600 & 277  & 23  & 1  &    &   \\
			-45 & 157288655 & 17615557 & 1976344 & 221746 & 24731 & 2679 & 281  & 23  & 4  &    &   \\
			\rowcolor[HTML]{E8E8E8} 
			-44 & 157474057 & 17640425 & 1980331 & 222635 & 24737 & 2670 & 265  & 25  & 1  &    &   \\
			-43 & 161440665 & 18131460 & 2041299 & 230497 & 25939 & 2801 & 319  & 34  & 3  &    &   \\
			\rowcolor[HTML]{E8E8E8} 
			-42 & 161477126 & 18139987 & 2042247 & 230677 & 26217 & 2963 & 323  & 35  & 4  & 1  &   \\
			-41 & 166326919 & 18758031 & 2124063 & 240927 & 27317 & 3107 & 338  & 41  & 5  & 0  &   \\
			\rowcolor[HTML]{E8E8E8} 
			-40 & 166066810 & 18715681 & 2117644 & 240503 & 27342 & 2987 & 352  & 38  & 2  & 0  &   \\
			-39 & 172416970 & 19515444 & 2219822 & 253386 & 28718 & 3232 & 321  & 34  & 2  & 0  &   \\
			\rowcolor[HTML]{E8E8E8} 
			-38 & 172061911 & 19471804 & 2214742 & 252503 & 28497 & 3286 & 363  & 37  & 3  & 0  &   \\
			-37 & 179636293 & 20424829 & 2335566 & 268516 & 30951 & 3562 & 378  & 45  & 2  & 0  &   \\
			\rowcolor[HTML]{E8E8E8} 
			-36 & 177640738 & 20172741 & 2304548 & 264324 & 30438 & 3401 & 382  & 42  & 4  & 0  &   \\
			-35 & 186402059 & 21262299 & 2443733 & 282221 & 32943 & 3845 & 429  & 56  & 7  & 0  &   \\
			\rowcolor[HTML]{E8E8E8} 
			-34 & 183557880 & 20917244 & 2399966 & 277126 & 32006 & 3774 & 441  & 50  & 3  & 0  &   \\
			-33 & 192595098 & 22050424 & 2542567 & 295787 & 34423 & 4026 & 458  & 50  & 6  & 0  &   \\
			\rowcolor[HTML]{E8E8E8} 
			-32 & 187933156 & 21473919 & 2470177 & 286212 & 33446 & 3806 & 454  & 56  & 7  & 1  &   \\
			-31 & 197513541 & 22680026 & 2625029 & 305443 & 35984 & 4232 & 477  & 56  & 7  & 0  &   \\
			\rowcolor[HTML]{E8E8E8} 
			-30 & 192806659 & 22085709 & 2551227 & 296524 & 35018 & 4006 & 466  & 46  & 7  & 0  &   \\
			-29 & 200679723 & 23085997 & 2679220 & 313792 & 37043 & 4416 & 538  & 61  & 6  & 0  &   \\
			\rowcolor[HTML]{E8E8E8} 
			-28 & 196747898 & 22596987 & 2618029 & 306978 & 35944 & 4255 & 526  & 68  & 14 & 1  &   \\
			-27 & 208373342 & 24066755 & 2807751 & 331861 & 39362 & 4796 & 620  & 66  & 6  & 0  &   \\
			\rowcolor[HTML]{E8E8E8} 
			-26 & 207524517 & 23975103 & 2798797 & 330124 & 39711 & 4907 & 595  & 71  & 9  & 2  &   \\
			-25 & 220660508 & 25634863 & 3011643 & 358793 & 43228 & 5235 & 652  & 72  & 12 & 1  &   \\
			\rowcolor[HTML]{E8E8E8} 
			-24 & 218679154 & 25398221 & 2985061 & 355705 & 42957 & 5207 & 683  & 95  & 13 & 2  &   \\
			-23 & 231838351 & 27072426 & 3201247 & 383263 & 46861 & 5759 & 698  & 93  & 10 & 0  &   \\
			\rowcolor[HTML]{E8E8E8} 
			-22 & 227001776 & 26473195 & 3125624 & 374095 & 45595 & 5669 & 717  & 72  & 8  & 0  &   \\
			-21 & 238581530 & 27942303 & 3315655 & 399252 & 49078 & 6132 & 772  & 93  & 9  & 0  &   \\
			\rowcolor[HTML]{E8E8E8} 
			-20 & 225379188 & 26275304 & 3104488 & 371187 & 45314 & 5658 & 734  & 101 & 12 & 1  &   \\
			-19 & 238927515 & 27983835 & 3319062 & 400185 & 49312 & 6217 & 745  & 98  & 12 & 4  &   \\
			\rowcolor[HTML]{E8E8E8} 
			-18 & 228438575 & 26674122 & 3153715 & 379285 & 46182 & 5764 & 753  & 99  & 14 & 1  &   \\
			-17 & 243203220 & 28525498 & 3391278 & 410922 & 50370 & 6350 & 821  & 124 & 21 & 2  &   \\
			\rowcolor[HTML]{E8E8E8} 
			-16 & 235237016 & 27535204 & 3268627 & 393297 & 48604 & 6120 & 770  & 113 & 20 & 1  &   \\
			-15 & 251141295 & 29529930 & 3521444 & 425681 & 52563 & 6507 & 855  & 108 & 13 & 0  &   \\
			\rowcolor[HTML]{E8E8E8} 
			-14 & 245657729 & 28854709 & 3434713 & 415428 & 51396 & 6518 & 811  & 115 & 17 & 3  & 1 \\
			-13 & 261697331 & 30871436 & 3694997 & 449396 & 56052 & 7063 & 941  & 133 & 16 & 1  & 0 \\
			\rowcolor[HTML]{E8E8E8} 
			-12 & 253325086 & 29845788 & 3566511 & 433328 & 53567 & 6944 & 953  & 137 & 18 & 5  & 0 \\
			-11 & 273796317 & 32442924 & 3903869 & 479799 & 59764 & 7739 & 1018 & 146 & 16 & 3  & 1 \\
			\rowcolor[HTML]{E8E8E8} 
			-10 & 274354895 & 32570409 & 3933248 & 483817 & 61117 & 7859 & 1090 & 158 & 22 & 3  & 1 \\
			-9  & 286644016 & 34099599 & 4120679 & 508154 & 64219 & 8408 & 1124 & 162 & 25 & 5  & 0 \\
			\rowcolor[HTML]{E8E8E8} 
			-8  & 285492760 & 34048670 & 4132705 & 512306 & 65401 & 8539 & 1218 & 157 & 30 & 3  & 1 \\
			-7  & 294625344 & 35172331 & 4274938 & 530226 & 67242 & 8802 & 1238 & 178 & 36 & 10 & 0 \\
			\rowcolor[HTML]{E8E8E8} 
			-6  & 288344140 & 34440842 & 4187603 & 519919 & 66531 & 8756 & 1258 & 179 & 32 & 10 & 2 \\
			-5  & 296011555 & 35390375 & 4307092 & 534847 & 68057 & 9124 & 1258 & 161 & 29 & 4  & 1 \\
			\rowcolor[HTML]{E8E8E8} 
			-4  & 286503794 & 34243792 & 4167061 & 518758 & 66184 & 8826 & 1225 & 179 & 22 & 1  & 0 \\
			-3  & 284220703 & 33864884 & 4102245 & 508054 & 64912 & 8449 & 1142 & 167 & 18 & 3  & 0 \\
			\rowcolor[HTML]{E8E8E8} 
			-2  & 282984904 & 33788571 & 4106387 & 509356 & 65280 & 8448 & 1188 & 181 & 28 & 8  & 2 \\
			-1  & 285790536 & 34071252 & 4131348 & 511701 & 65077 & 8515 & 1203 & 163 & 20 & 4  & 1 \\
			\rowcolor[HTML]{E8E8E8} 
			0   & 285132051 & 34090579 & 4150019 & 517031 & 66320 & 8805 & 1190 & 170 & 23 & 8  & 1 \\
			1   & 285816973 & 34079474 & 4134226 & 512001 & 64923 & 8454 & 1166 & 180 & 22 & 4  & 0 \\
			\rowcolor[HTML]{E8E8E8} 
			2   & 282965478 & 33790755 & 4105123 & 510733 & 65120 & 8741 & 1243 & 164 & 24 & 4  & 1 \\
			3   & 284247018 & 33856898 & 4102916 & 508198 & 64138 & 8292 & 1171 & 182 & 29 & 5  & 2 \\
			\rowcolor[HTML]{E8E8E8} 
			4   & 286532237 & 34254862 & 4168340 & 517632 & 66368 & 8872 & 1297 & 203 & 44 & 7  & 1 \\
			5   & 296029198 & 35386325 & 4304568 & 534186 & 68239 & 8887 & 1188 & 172 & 19 & 3  & 1 \\
			\rowcolor[HTML]{E8E8E8} 
			6   & 288364738 & 34451448 & 4192449 & 522034 & 66343 & 8860 & 1169 & 166 & 29 & 6  & 1 \\
			7   & 294642143 & 35177434 & 4275015 & 530500 & 67604 & 8935 & 1228 & 168 & 24 & 6  & 1 \\
			\rowcolor[HTML]{E8E8E8} 
			8   & 285494550 & 34054118 & 4132141 & 513234 & 65554 & 8650 & 1166 & 152 & 27 & 5  & 1 \\
			9   & 286594727 & 34088658 & 4124488 & 508351 & 64462 & 8372 & 1148 & 169 & 30 & 7  & 1 \\
			\rowcolor[HTML]{E8E8E8} 
			10  & 274348420 & 32567912 & 3931836 & 484093 & 61559 & 7998 & 1120 & 152 & 26 & 3  & 1 \\
			11  & 273777740 & 32426966 & 3899784 & 477812 & 60258 & 7824 & 1110 & 164 & 20 & 1  & 0 \\
			\rowcolor[HTML]{E8E8E8} 
			12  & 253311421 & 29839307 & 3566489 & 433427 & 53589 & 6795 & 820  & 108 & 12 & 1  & 0 \\
			13  & 261679280 & 30866198 & 3694130 & 450241 & 55652 & 6897 & 898  & 124 & 24 & 5  & 1 \\
			\rowcolor[HTML]{E8E8E8} 
			14  & 245666156 & 28870503 & 3440583 & 416527 & 51179 & 6679 & 894  & 129 & 22 & 2  & 1 \\
			15  & 251104118 & 29523322 & 3520513 & 426857 & 52579 & 6798 & 862  & 110 & 15 & 2  &   \\
			\rowcolor[HTML]{E8E8E8} 
			16  & 235222471 & 27532894 & 3263211 & 392940 & 48632 & 6059 & 786  & 93  & 9  & 0  &   \\
			17  & 243217035 & 28529132 & 3394144 & 409862 & 50337 & 6288 & 822  & 101 & 18 & 2  &   \\
			\rowcolor[HTML]{E8E8E8} 
			18  & 228433005 & 26670900 & 3154649 & 378424 & 46025 & 5816 & 738  & 97  & 14 & 2  &   \\
			19  & 238932915 & 27978916 & 3321718 & 399985 & 49124 & 6348 & 795  & 101 & 9  & 1  &   \\
			\rowcolor[HTML]{E8E8E8} 
			20  & 225357244 & 26282933 & 3105430 & 372688 & 45379 & 5720 & 683  & 92  & 20 & 0  &   \\
			21  & 238581056 & 27934903 & 3314769 & 398874 & 49033 & 6267 & 801  & 106 & 10 & 3  &   \\
			\rowcolor[HTML]{E8E8E8} 
			22  & 227016320 & 26480081 & 3126879 & 374803 & 45568 & 5598 & 702  & 115 & 15 & 1  &   \\
			23  & 231864937 & 27070736 & 3199074 & 383370 & 46714 & 5774 & 756  & 103 & 9  & 1  &   \\
			\rowcolor[HTML]{E8E8E8} 
			24  & 218674999 & 25405730 & 2986737 & 356172 & 43085 & 5357 & 638  & 78  & 10 & 2  &   \\
			25  & 220654476 & 25636190 & 3013633 & 358464 & 43490 & 5245 & 656  & 63  & 13 & 0  &   \\
			\rowcolor[HTML]{E8E8E8} 
			26  & 207550224 & 23975565 & 2797734 & 329791 & 39477 & 4686 & 556  & 66  & 6  & 0  &   \\
			27  & 208353637 & 24063968 & 2807640 & 330900 & 39448 & 4718 & 613  & 74  & 10 & 0  &   \\
			\rowcolor[HTML]{E8E8E8} 
			28  & 196716398 & 22593961 & 2621784 & 306740 & 36422 & 4337 & 502  & 49  & 6  & 2  &   \\
			29  & 200688403 & 23081340 & 2678062 & 313512 & 37251 & 4541 & 539  & 79  & 9  & 1  &   \\
			\rowcolor[HTML]{E8E8E8} 
			30  & 192822269 & 22086684 & 2552125 & 297179 & 34594 & 4073 & 467  & 52  & 4  & 0  &   \\
			31  & 197517587 & 22669608 & 2621551 & 305250 & 36005 & 4255 & 493  & 61  & 8  & 0  &   \\
			\rowcolor[HTML]{E8E8E8} 
			32  & 187932603 & 21464007 & 2468752 & 286987 & 33400 & 3951 & 494  & 60  & 4  & 0  &   \\
			33  & 192596736 & 22046542 & 2542733 & 296132 & 34594 & 4077 & 458  & 48  & 1  & 0  &   \\
			\rowcolor[HTML]{E8E8E8} 
			34  & 183600374 & 20928566 & 2398187 & 275541 & 31768 & 3651 & 440  & 60  & 5  & 0  &   \\
			35  & 186373340 & 21266926 & 2441600 & 282075 & 32839 & 3936 & 451  & 45  & 1  & 0  &   \\
			\rowcolor[HTML]{E8E8E8} 
			36  & 177654153 & 20160128 & 2302012 & 264274 & 30523 & 3503 & 392  & 45  & 6  & 1  &   \\
			37  & 179653033 & 20422240 & 2334627 & 267943 & 30676 & 3620 & 403  & 36  & 4  & 0  &   \\
			\rowcolor[HTML]{E8E8E8} 
			38  & 172064515 & 19474489 & 2214595 & 252591 & 28745 & 3279 & 338  & 38  & 5  & 0  &   \\
			39  & 172423860 & 19521029 & 2221306 & 253670 & 28972 & 3311 & 373  & 37  & 5  & 0  &   \\
			\rowcolor[HTML]{E8E8E8} 
			40  & 166104905 & 18724906 & 2117971 & 240611 & 27344 & 3091 & 371  & 43  & 3  & 0  &   \\
			41  & 166323319 & 18749509 & 2123150 & 240448 & 27089 & 3053 & 338  & 47  & 4  & 0  &   \\
			\rowcolor[HTML]{E8E8E8} 
			42  & 161484723 & 18142330 & 2044048 & 230015 & 25900 & 2796 & 310  & 29  & 1  & 0  &   \\
			43  & 161411134 & 18136861 & 2044972 & 230317 & 26167 & 2940 & 338  & 30  & 3  & 1  &   \\
			\rowcolor[HTML]{E8E8E8} 
			44  & 157504079 & 17645743 & 1979344 & 222051 & 24488 & 2675 & 286  & 27  & 3  &    &   \\
			45  & 157289722 & 17615572 & 1977617 & 222128 & 24897 & 2697 & 278  & 27  & 1  &    &   \\
			\rowcolor[HTML]{E8E8E8} 
			46  & 154924957 & 17311332 & 1937595 & 216315 & 23849 & 2569 & 266  & 18  & 1  &    &   \\
			47  & 154898752 & 17306039 & 1937480 & 216551 & 23822 & 2619 & 299  & 25  & 3  &    &   \\
			\rowcolor[HTML]{E8E8E8} 
			48  & 152317224 & 16983058 & 1896656 & 211348 & 23314 & 2505 & 236  & 34  & 3  &    &   \\
			49  & 150644552 & 16780039 & 1867095 & 207637 & 22946 & 2471 & 238  & 21  & 3  &    &   \\
			\rowcolor[HTML]{E8E8E8} 
			50  & 147271752 & 16360151 & 1817438 & 200919 & 22107 & 2344 & 225  & 20  & 0  &    &   \\
			51  & 144377149 & 16013063 & 1776781 & 195943 & 21361 & 2212 & 224  & 11  & 1  &    &   \\
			\rowcolor[HTML]{E8E8E8} 
			52  & 140821922 & 15572918 & 1723600 & 189431 & 20598 & 2164 & 207  & 23  & 0  &    &   \\
			53  & 137514881 & 15179551 & 1672906 & 183860 & 20090 & 2121 & 213  & 23  & 1  &    &   \\
			\rowcolor[HTML]{E8E8E8} 
			54  & 134230989 & 14783335 & 1627558 & 177736 & 19311 & 2036 & 199  & 21  & 1  &    &   \\
			55  & 131236967 & 14423710 & 1582500 & 172696 & 18457 & 1861 & 175  & 17  & 1  &    &   \\
			\rowcolor[HTML]{E8E8E8} 
			56  & 128906212 & 14137741 & 1546258 & 168321 & 17967 & 1829 & 186  & 13  & 1  &    &   \\
			57  & 126912794 & 13892042 & 1514923 & 163999 & 17368 & 1781 & 178  & 14  & 0  &    &   \\
			\rowcolor[HTML]{E8E8E8} 
			58  & 125903469 & 13766302 & 1499470 & 161793 & 17046 & 1688 & 158  & 10  & 0  &    &   \\
			59  & 125423037 & 13700814 & 1490810 & 160879 & 17254 & 1804 & 152  & 11  & 1  &    &   \\
			\rowcolor[HTML]{E8E8E8} 
			60  & 124678744 & 13607074 & 1477998 & 159281 & 16989 & 1674 & 156  & 22  & 2  &    &   \\ \bottomrule
		\end{longtable}
	\end{center}
\end{footnotesize}

\noindent{\bf Acknowledgement.} The authors thank J\'anos Pintz and Joseph Najnudel for some useful discussions. ITK was financially supported by a \emph{Leverhulme Trust Research Project Grant} (RPG-2019-034).


\bibliographystyle{abbrv}
\bibliography{prime_biblio}

\end{document}